\documentclass[english]{article}
\usepackage[T1]{fontenc}
\usepackage[latin9]{inputenc}
\usepackage{geometry}
\usepackage{amsmath}
\usepackage{amsthm}
\usepackage{amssymb}
\usepackage{graphicx}
\usepackage{wasysym}

\makeatletter
\numberwithin{equation}{section}
\numberwithin{figure}{section}
\theoremstyle{plain}
\newtheorem{thm}{\protect\theoremname}
  \theoremstyle{definition}
  \newtheorem{defn}[thm]{\protect\definitionname}
  \theoremstyle{remark}
  \newtheorem*{rem*}{\protect\remarkname}
  \theoremstyle{plain}
  \newtheorem{lem}[thm]{\protect\lemmaname}
  \theoremstyle{remark}
  \newtheorem{claim}[thm]{\protect\claimname}
  \theoremstyle{remark}
  \newtheorem*{acknowledgement*}{\protect\acknowledgementname}

\usepackage{hyperref}
\hypersetup{
    colorlinks,
    allcolors=red
}
\usepackage[lined,boxed,ruled,norelsize,algo2e]{algorithm2e}
\@addtoreset{section}{part}
\usepackage{amsfonts}

\makeatother

\usepackage{babel}
  \providecommand{\acknowledgementname}{Acknowledgement}
  \providecommand{\claimname}{Claim}
  \providecommand{\definitionname}{Definition}
  \providecommand{\lemmaname}{Lemma}
  \providecommand{\remarkname}{Remark}
\providecommand{\theoremname}{Theorem}

\begin{document}
\global\long\def\defeq{\stackrel{\mathrm{{\scriptscriptstyle def}}}{=}}
\global\long\def\norm#1{\left\Vert #1\right\Vert }
\global\long\def\R{\mathbb{R}}
\global\long\def\Ent{\mathrm{Ent}}
 \global\long\def\Rn{\mathbb{R}^{n}}
\global\long\def\tr{\mathrm{Tr}}
\global\long\def\diag{\mathrm{diag}}
\global\long\def\cov{\mathrm{Cov}}
\global\long\def\E{\mathbb{E}}
\global\long\def\P{\mathbb{P}}
\global\long\def\Var{\mbox{Var}}
\global\long\def\vol{\mbox{vol}}
\global\long\def\rank{\text{rank}}
\global\long\def\lref#1{\text{Lem }\ltexref{#1}}
\global\long\def\lreff#1#2{\text{Lem }\ltexref{#1}.\ltexref{#1#2}}
\global\long\def\ltexref#1{\ref{lem:#1}}\global\long\def\ttag#1{\tag{#1}}
\global\long\def\cirt#1{\raisebox{.5pt}{\textcircled{\raisebox{-.9pt}{#1}}}}
\global\long\def\spe{\mathrm{op}}

\title{Stochastic Localization + Stieltjes Barrier = Tight Bound for Log-Sobolev}

\author{Yin Tat Lee\thanks{Microsoft Research and University of Washington, yile@microsoft.com and yintat@uw.edu},
Santosh S. Vempala\thanks{Georgia Tech, vempala@gatech.edu}}
\maketitle
\begin{abstract}
Logarithmic Sobolev inequalities are a powerful way to estimate the
rate of convergence of Markov chains and to derive concentration inequalities
on distributions. We prove that the log-Sobolev constant of any isotropic
logconcave density in $\R^{n}$ with support of diameter $D$ is $\Omega(1/D)$,
resolving a question posed by Frieze and Kannan in 1997. This is asymptotically
the best possible estimate and improves on the previous bound of $\Omega(1/D^{2})$
by Kannan-Lov\'{a}sz-Montenegro. It follows that for any isotropic logconcave
density, the ball walk with step size $\delta=\Theta(1/\sqrt{n})$
mixes in $O\left(n^{2}D\right)$ proper steps from \emph{any }starting
point. This improves on the previous best bound of $O(n^{2}D^{2})$
and is also asymptotically tight. The new bound leads to the following
refined large deviation inequality for an $L$-Lipschitz function $g$
over an isotropic logconcave density $p$: for any $t>0$, 
\[
\P_{x\sim p}\left(\left|g(x)-\bar{g}\right|\geq c\cdot L\cdot t\right)\leq\exp(-\frac{t^{2}}{t+\sqrt{n}})
\]
where $\bar{g}$ is the median or mean of $g$ for $x\sim p$; this
generalizes/improves on previous bounds by Paouris and by Guedon-Milman.
The technique also bounds the ``small ball'' probability in terms
of the Cheeger constant, and recovers the current best bound. Our
main proof is based on stochastic localization together with a Stieltjes-type
barrier function.
\end{abstract}

\section{Introduction}

This purpose of this paper is to understand the asymptotic behavior
of the log-Sobolev and log-Cheeger constants of convex bodies and
logconcave distributions in $\R^{n}$. These fundamental parameters,
which we will define presently, have many important connections and
applications (cf. \cite{Ledoux1999}). To introduce them, we first
remind the reader of the Cheeger constant (a.k.a. isoperimetric constant
or expansion).
\begin{defn}
For a density $p$ in $\R^{n}$, the Cheeger constant of $p$ is defined
as
\[
\psi_{p}\defeq\inf_{S\subseteq\R^{n}}\frac{\int_{\partial S}p(x)dx}{\min\left\{ \int_{S}p(s)dx,\int_{\R^{n}\setminus S}p(x)dx\right\} }.
\]
\end{defn}

Kannan, Lovász and Simonovits \cite{KLS95} conjectured that for any
logconcave density, the Cheeger constant satisfies\footnote{We write $\gtrsim$to denote ``at least a constant times''.}
$\psi_{p}\gtrsim\norm A_{\spe}^{-1/2}$ where $A$ is the covariance
matrix of $p$. A density/distribution is called \emph{isotropic }if
its covariance matrix is the identity, a normalization that can be
achieved via an affine transformation. For isotropic logconcave densities,
the conjecture says the Cheeger constant is $\Omega(1)$. The current
best estimate of $\psi_{p}$ is the following recent result.
\begin{thm}[\cite{LeeV17KLS}]
\label{thm:LeeV} For any logconcave density $p$ in $\R^{n}$ with
covariance matrix $A$,
\[
\psi_{p}\gtrsim\left(\tr\left(A^{2}\right)\right)^{-1/4}.
\]
\end{thm}

For isotropic $p$, this gives a bound of $\psi_{p}\gtrsim n^{-\frac{1}{4}}$.
The KLS hyperplane conjecture plays a central role in asymptotic convex
geometry, and implies several other well-known (older) conjectures,
including slicing (or hyperplane) and thin-shell (or variance) conjectures,
the Poincare conjecture, central limit, exponential concentration
etc. (see e.g., \cite{brazitikos2014geometry}). 

The KLS conjecture was motivated by the study of the convergence of
a Markov chain, the \emph{ball walk} in a convex body. To sample uniformly
from a convex body, the ball walk starts at some point in the body,
picks a random point in the ball of radius $\delta$ around the current
point and if the chosen point is in the body, it steps to the new
point. It can be generalized to sampling any logconcave density by
using a Metropolis filter. As shown in \cite{KLS97}, the ball walk
applied to a logconcave density mixes in $O^{*}(n^{2}/\psi_{p}^{2})$
steps from a warm start, which using the current-best bound \cite{LeeV17KLS}
is $O^{*}(n^{2.5})$. Looking closer, from a starting distribution
$Q_{o}$, the distance of the distribution obtained after $t$ steps
from $Q_{o}$ to the stationary distribution $Q$ drops as 
\[
d(Q_{t},Q)\le d(Q_{0},Q)\left(1-\frac{\phi^{2}}{2}\right)^{t}
\]
where $\phi$ is the conductance of the Markov chain and $d(.,.)$
is the $\chi$-squared distance. The conductance can be viewed as
the Cheeger constant of the Markov chain. Thus the number of steps
needed is $O(\phi^{-2}\log(1/d(Q_{0},Q)))$. Roughly speaking, for
the ball walk applied to a logconcave density $p$, the conductance
is $\Omega(\psi_{p}/n)$, leading to the bound of $O^{*}(n^{2}/\psi_{p}^{2})$
steps from a warm start. The dependence on the starting distribution
leads to an additional factor of $n$ in the running time when the
starting distribution is not warm (i.e., $d(Q_{1},Q)$ after one step
can be $e^{\tilde{\Omega}(n)}$ ). This is a general issue for Markov
chains. One way to address this is via the log-Sobolev constant \cite{diaconis1996,Ledoux1999}.
We first define it for a density, then for a Markov chain.
\begin{defn}
For a density $p$ the log-Sobolev constant $\rho_{p}$ is the largest
$\rho$ such that for every smooth function $f:\R^{n}\rightarrow\R$
with $\int f^{2}dp=1$, we have
\[
\rho\le\frac{2\int\norm{\nabla f}^{2}dp}{\int f^{2}\log f^{2}dp}.
\]
\end{defn}

A closely related parameter is the following.
\begin{defn}
The \emph{log-Cheeger }or log-isoperimetric constant $\kappa_{p}$
of a density $p$ in $\R^{n}$ is
\end{defn}

\[
\kappa_{p}=\inf_{S\subseteq\R^{n}}\frac{p(\partial S)}{\min\left\{ p(S),p(\R^{n}\setminus S)\right\} \sqrt{\log\left(\frac{1}{p(S)}\right)}}.
\]
It is known that $\rho_{p}=\Theta(\kappa_{p}^{2})$ (see e.g., \cite{ledoux1994simple}).
The log-Cheeger constant shows more explicitly that the log-Sobolev
constant is a uniform bound on the expansion ``at every scale''. 

For a reversible Markov chain with transition operator $P$ and stationary
density $Q$, the analogous definition is the infimum over all smooth
functions satisfying $f:\R^{n}\rightarrow\R$ with $\int f^{2}dp=1$
of 
\[
\rho(P)=\frac{\int_{x}\int_{y}\norm{f(x)-f(y)}^{2}P(x,y)dQ(x)}{\int f^{2}\log f^{2}dQ}.
\]
Diaconis and Saloff-Coste \cite{diaconis1996} show that the distribution
after $t$ steps satisfies $\Ent(Q_{t})\le e^{-4\rho(P)t}\Ent(Q_{0})$
where $\Ent(Q_{t})=\int Q_{t}\log\frac{Q_{t}}{Q_{0}}dQ_{0}$ is the
entropy with respect to the stationary distribution. Thus, the dependence
of the mixing time on the starting distribution goes down from $\log(1/d(Q_{0},Q))$
to $\log\log(1/d(Q_{0},Q)))$. Moreover, just as in the case of the
Cheeger constant, for the ball walk, the Markov chain parameter is
determined by the log-Sobelov constant $\rho_{p}$ for sampling from
the density $p$. It is thus natural to ask for the best possible
bound on $\rho_{p}$ or $\kappa_{p}$. Unlike the Cheeger constant,
which is conjectured to be at least a constant for isotropic logconcave
densities, it is known that $\rho_{p}$ cannot be bounded from below
by a universal constant, in particular for distributions that are
not ``$\psi_{2}$'' (distributions with sub-Gaussian tail).

Kannan, Lovász and Montenegro \cite{KannanLM06} gave the following
bound on $\kappa_{p}$. Our main result (Theorem \ref{thm:logsob})
is an improvement of this bound to the best possible. 
\begin{thm}[\cite{KannanLM06}]
For an isotropic logconcave density $K\subset\R^{n}$ with support
of diameter $D$, $we$ have $\kappa_{p}\gtrsim\frac{1}{D}$ and $\rho_{p}\gtrsim\frac{1}{D^{2}}$.
\end{thm}

From the above bound, it follows that the ball walk mixes in $O\left(n^{2}D^{2}\right)$
\emph{proper} steps of step size $\delta$. A proper step is one where
the current point changes. For an isotropic logconcave density $\delta=\Theta(\frac{1}{\sqrt{n}})$
is small enough so that the number of wasted steps is of the same
order as the number of proper steps in expectation. Moreover, by restricting
to a ball of radius $D=O(\sqrt{n})$, the resulting distribution remains
near-isotropic and very close in total variation distance to the original.
Together, this considerations imply a bound of $O^{*}(n^{3})$ proper
steps from any starting point as shown in \cite{KannanLM06}. Is this
bound the best possible? From a warm start, the KLS conjecture implies
a bound of $O^{*}(n^{2})$ steps and current best bound is $O^{*}(n^{2.5})$.
Thus, the mixing of the ball walk, which was the primary motivation
for formulation of the KLS conjecture, also provides a compelling
reason to study the log-Sobolev constant. Estimating the log-Sobolev
constant was posed as an open problem by Frieze and Kannan \cite{frieze1999}
when they analyzed the log-Sobolev constant of the grid walk to sample
sufficiently smooth logconcave densities.

The Cheeger and log-Sobolev constants also play an important role
in the phenomenon known as concentration of measure. The following
result is due to Gromov and Milman. 
\begin{thm}[Lipschitz concentration \cite{GromovM83}]
 For any $L$-Lipschitz function $g$ in $\R^{n},$ and isotropic
logconcave density $p$, 
\[
\P_{x\sim p}\left(\left|g(x)-\E g\right|\ge c\cdot L\cdot t\right)\le e^{-t\psi_{p}}.
\]
\end{thm}

Using the best-known estimate of the Cheeger constant gives a bound
of $e^{-t/n^{1/4}}$ \cite{LeeV17KLS}. A different bound, independent
of the Cheeger constant, for the deviation in length of a random vector
was given in a celebrated paper by Paouris \cite{Paouris2006} and
improved by Guedon and Milman \cite{GuedonM11} (Paouris' result has
only the second term in the minimum below, and is sharp when $t\gtrsim\sqrt{n}$). 
\begin{thm}[\cite{GuedonM11,Paouris2006}]
 For any isotropic logconcave density $p$,
\[
\P_{x\sim p}\left(\left|\norm x-\sqrt{n}\right|\ge c\cdot t\right)\le e^{-\min\left\{ \frac{t^{3}}{n},t\right\} }.
\]
\end{thm}

Our tight log-Sobolev bound will be useful in proving an improved
concentration inequality. 

\subsection{Results}

Our main theorem is the following.
\begin{thm}
\label{thm:logsob}For any isotropic logconcave density $p$ with
support of diameter $D$, the log-Cheeger constant satisfies $\kappa_{p}\gtrsim\frac{1}{\sqrt{D}}$
and the log-Sobolev constant satisfies $\rho_{p}\gtrsim\frac{1}{D}$. 
\end{thm}

As we show in Section \ref{sec:tight}, these bounds are the best
possible (Lemma \ref{lem:lower_bound_logsob}). The improved bound
has interesting consequences. The first is an improved concentration
of mass inequality for logconcave densities. In particular, this gives
an alternative proof of Paouris' (optimal) inequality \cite{Paouris2006}
for the large deviation case ($t\ge\sqrt{n})$.
\begin{thm}
\label{thm:conc}For any $L$-Lipschitz function $g$ in $\Rn$ and
any isotropic logconcave density $p$, we have that
\[
\P_{x\sim p}\left(\left|g(x)-\bar{g}(x)\right|\geq c\cdot L\cdot t\right)\leq\exp(-\frac{t^{2}}{t+\sqrt{n}})
\]
where $\bar{g}$ is the median or mean of $g(x)$ for $x\sim p$. 
\end{thm}

For the Euclidean norm, this gives 

\[
\P_{x\sim p}\left(\norm x\geq c\cdot t\cdot\sqrt{n}\right)\leq\exp(-\min\left\{ t,t^{2}\right\} \sqrt{n}).
\]
As mentioned earlier, the previous best bound was $\exp(-\min\left\{ t,t^{3}\right\} \sqrt{n})$
\cite{GuedonM11} for the Euclidean norm and $\exp(-\frac{t}{n^{1/4}})$
for a general Lipschitz function $g$ \cite{LeeV17KLS}. The new bound
can be viewed as an improvement and generalization of both. Also this
concentration result does not need bounded support for the density
$p$. 

Next we bound the small ball probability in terms of the Cheeger constant.
\begin{thm}
\label{thm:small-ball}Assume that for $k\geq1$, $\psi_{p}\gtrsim\left(\tr A^{k}\right)^{-1/2k}$
for all logconcave distribution $p$ with covariance $A$. Then, for
any isotropic logconcave distribution $p$ and any $0\leq\varepsilon\leq c_{1}$,
we have that
\[
\P_{x\sim p}\left(\norm x_{2}^{2}\leq\varepsilon n\right)\leq\varepsilon^{c_{2}k^{-1}n^{1-1/k}}
\]
for some universal constant $c_{1}$ and $c_{2}$.
\end{thm}

\begin{rem*}
The current best KLS estimate gives $\P_{x\sim p}\left(\norm x_{2}^{2}\leq\varepsilon n\right)\leq\varepsilon^{c_{2}\sqrt{n}}$,
which recovers the current best small ball estimate, also due to Paouris
\cite{paouris2012small}.
\end{rem*}
As another consequence, we circle back to the analysis of the ball
walk to resolve the open problem posed by Frieze and Kannan \cite{frieze1999}.
\begin{thm}
The ball walk with step size $\delta=\Theta(1/\sqrt{n})$ applied
to an isotropic logconcave density in $\R^{n}$ with support of diameter
$D$ mixes in $O^{*}(n^{2}D)$ proper steps from any starting point
of positive density (or from any interior point of a convex body).
\end{thm}

The choice of $\delta=\Theta\left(1/\sqrt{n}\right)$ is the best
possible for isotropic logconcave distributions (Lemma \ref{lem:lower_bound_ball}).
The bound on the number of steps improves on the previous best bound
of $O^{*}(n^{2}D^{2})$ proper steps for the mixing of the ball walk
from an arbitrary starting point \cite{KannanLM06} and as we show
in Section \ref{sec:tight}, $O(n^{2}D)$ is the best possible bound.
For sampling, we can restrict the density to a ball of radius $O(\sqrt{n})$
losing only a negligibly small measure, so the bound is $O(n^{2.5})$
from an arbitrary starting point, which matches the current best bound
from a warm start. The mixing time from a warm start for an isotropic
logconcave density is $O(n^{2}\psi_{p}^{2})$, or $O(n^{2})$ if the
KLS conjecture is true; but from an arbitrary start, our analysis
is essentially the best possible, independent of any further progress
on the conjecture! 

\section{Approach: Stochastic localization}

In this section, we describe the stochastic localization method introduced
by Eldan \cite{Eldan2013}, and, in particular, the variant of the
method used in \cite{LeeV17KLS}. The idea of the approach is to gradually
modify the density $p$ by making infinitesimal changes so that the
measure of a set and of its boundary is not changed by much, but the
density function itself accumulates a significant Gaussian component,
i.e., it looks like a Gaussian density function times a logconcave
function. For such densities, we can use standard localization (or
other methods) to prove that the log-Sobolev constant is large. While
this is the same high-level approach as in previous papers, several
new challenges arise. First, unlike previous applications, we cannot
simply work with subsets of measure $1/2$ or a constant, it is crucial
to consider arbitrarily small subsets. Second, as we will discuss
presently, we need a more refined potential function to understand
the evolution of the measure. 
\begin{defn}
\label{def:A}Given a logconcave distribution $p$, we define the
following stochastic differential equation:
\begin{equation}
c_{0}=0,\quad dc_{t}=dW_{t}+\mu_{t}dt,\label{eq:dBt}
\end{equation}
where the probability distribution $p_{t}$, the mean $\mu_{t}$ and
the covariance $A_{t}$ are defined by
\[
p_{t}(x)=\frac{e^{c_{t}^{T}x-\frac{t}{2}\norm x_{2}^{2}}p(x)}{\int_{\Rn}e^{c_{t}^{T}y-\frac{t}{2}\norm y_{2}^{2}}p(y)dy},\quad\mu_{t}=\E_{x\sim p_{t}}x,\quad A_{t}=\E_{x\sim p_{t}}(x-\mu_{t})(x-\mu_{t})^{T}.
\]
\end{defn}

We collect the properties of this stochastic localization that we
will use in this paper in the following Lemma.
\begin{lem}[\cite{LeeV17KLS}]
\label{lem:dp_dA}For any logconcave distribution $p$ with bounded
support, the stochastic process $p_{t}$ defined in Definition \ref{def:A}
exists and is unique. Also, $p_{t}$ is a martingale, and in particular,
for any $x\in\R^{n}$ 
\[
dp_{t}(x)=(x-\mu_{t})^{T}dW_{t}\cdot p_{t}(x).
\]
 Its covariance matrix satisfies
\begin{equation}
dA_{t}=\int_{\Rn}(x-\mu_{t})(x-\mu_{t})^{T}((x-\mu_{t})^{T}dW_{t})\cdot p_{t}(x)dx-A_{t}^{2}dt.\label{eq:dA}
\end{equation}
\end{lem}

In previous papers \cite{Eldan2013,LeeV17KLS}, the spectral norm
of the covariance the $\norm{A_{t}}_{\spe}$ is bounded via a potential
function of the form $\tr\left(A_{t}^{q}\right)$. However, to obtain
a tight result without extraneous logarithmic factors, we study the
Stieltjes-type potential $\tr\left((uI-A_{t})^{-q}\right)$. 

To define the potential, fix integers $n,q\geq1$ and a positive number
$\Phi>0$. Let $u(X)$ be the real-valued function on $n\times n$
symmetric matrices defined by the solution of the following equation
\begin{equation}
\tr((uI-X)^{-q})=\Phi\text{ and }X\preceq uI\label{eq:def_u}
\end{equation}
Note that this is the same as the solution to $\sum_{i=1}^{n}\frac{1}{(u-\lambda_{i})^{q}}=\Phi$
and $\lambda_{i}\leq u$ for all $i$ where $\lambda_{i}$ are the
eigenvalues of $X$. Similar potentials have been used to analyze
empirical covariance estimation \cite{srivastava2013covariance},
to build graph sparsifiers \cite{BSS12,allen2015spectral,lee2015constructing,lee2017sdp}
and to solve bandit problems \cite{audibert2013regret}.

The proof has the following ingredients: 
\begin{enumerate}
\item We show that for time $t$ up to $O(n^{-\frac{1}{2}})$, the spectral
norm of the covariance stays bounded (by a constant, say $2$) with
large probability (Lemma \ref{lem:norm_At}). This requires the use
of the Stieltjes-type potential function. 
\item Then we consider any measurable subset $S$, with $g_{0}=p_{0}(S)$
and analyze its measure at time $t$, i.e., $g_{t}=p_{t}(S)$. In
particular we show that up to time $\left(\log g_{0}+D\right)^{-1}$,
the expectation of $g_{t}\sqrt{\log(1/g_{t})}$ remains large, i.e.,
a constant factor times its initial value (Lemma \ref{lem:g_sqrt_g}). 
\item The density at time $t$ has a Gaussian component of variance $1/t$.
For such a distribution, the log-Cheeger constant is $\Omega(\sqrt{t})$
(Theorem \ref{thm:Gaussian-iso}).
\end{enumerate}
Together these facts will imply the main theorem. 

\section{Preliminaries}

In this section, we review some basic definitions and theorems that
we use in the proofs. 

\subsection{Stochastic calculus}

Given real-valued stochastic processes $x_{t}$ and $y_{t}$, the
quadratic variations $[x]_{t}$ and $[x,y]_{t}$ are real-valued stochastic
processes defined by
\[
[x]_{t}=\lim_{|P|\rightarrow0}\sum_{n=1}^{\infty}\left(x_{\tau_{n}}-x_{\tau_{n-1}}\right)^{2}\quad\text{and}\quad[x,y]_{t}=\lim_{|P|\rightarrow0}\sum_{n=1}^{\infty}\left(x_{\tau_{n}}-x_{\tau_{n-1}}\right)\left(y_{\tau_{n}}-y_{\tau_{n-1}}\right),
\]
where $P=\{0=\tau_{0}\leq\tau_{1}\leq\tau_{2}\leq\cdots\uparrow t\}$
is a stochastic partition of the non-negative real numbers, $|P|=\max_{n}\left(\tau_{n}-\tau_{n-1}\right)$
is called the \emph{mesh} of $P$ and the limit is defined using convergence
in probability. Note that $[x]_{t}$ is non-decreasing with $t$ and
$[x,y]_{t}$ can be defined as
\[
[x,y]_{t}=\frac{1}{4}\left([x+y]_{t}-[x-y]_{t}\right).
\]
If the processes $x_{t}$ and $y_{t}$ satisfy the SDEs $dx_{t}=\mu(x_{t})dt+\sigma(x_{t})dW_{t}$
and $dy_{t}=\nu(y_{t})dt+\eta(y_{t})dW_{t}$ where $W_{t}$ is a Wiener
process, we have that $[x]_{t}=\int_{0}^{t}\sigma^{2}(x_{s})ds$ and
$[x,y]_{t}=\int_{0}^{t}\sigma(x_{s})\eta(y_{s})ds$ and $d[x,y]_{t}=\sigma(x_{t})\eta(y_{t})dt$;
for a vector-valued SDE $dx_{t}=\mu(x_{t})dt+\Sigma(x_{t})dW_{t}$
and $dy_{t}=\nu(y_{t})dt+M(y_{t})dW_{t}$, we have that $[x^{i},x^{j}]_{t}=\int_{0}^{t}(\Sigma(x_{s})\Sigma^{T}(x_{s}))_{ij}ds$
and $d[x^{i},y^{j}]_{t}=(\Sigma(x_{t})M^{T}(y_{t}))_{ij}dt$.
\begin{lem}[Itô's formula]
\label{lem:Ito} Let $x$ be a semimartingale and $f$ be a twice
continuously differentiable function, then
\[
df(x_{t})=\sum_{i}\frac{df(x_{t})}{dx^{i}}dx^{i}+\frac{1}{2}\sum_{i,j}\frac{d^{2}f(x_{t})}{dx^{i}dx^{j}}d[x^{i},x^{j}]_{t}.
\]
\end{lem}

The next two lemmas are well-known facts about Wiener processes; first
the reflection principle.
\begin{lem}[Reflection principle]
\label{lem:reflection}Given a Wiener process $W(t)$ and $a,t\geq0$,
then we have that
\[
\P(\sup_{0\leq s\leq t}W(s)\geq a)=2\P(W(t)\geq a).
\]
\end{lem}

Second, a decomposition lemma for continuous martingales.
\begin{thm}[Dambis, Dubins-Schwarz theorem]
\label{thm:Dubins}Every continuous local martingale $M_{t}$ is
of the form
\[
M_{t}=M_{0}+W_{[M]_{t}}\text{ for all }t\geq0
\]
where $W_{s}$ is a Wiener process. 
\end{thm}

\subsection{Logconcave functions and isoperimetry}

We say a logconcave distribution is \emph{nearly} isotropic if its
covariance matrix $A$ satisfies $\Omega(1)\cdot I\preceq A\preceq O(1)\cdot I$.
The following lemma about logconcave densities is folklore, see e.g.,
\cite{Lovasz2007}.
\begin{lem}[Logconcave moments]
\label{lem:lcmom} Given a logconcave density $p$ in $\R^{n}$,
and any integer $k\geq1$, 
\[
\E_{x\sim p}\norm x^{k}\le(2k)^{k}\left(\E_{x\sim p}\norm x^{2}\right)^{k/2}.
\]
\end{lem}

\begin{thm}[Poincaré constant \cite{Mazja60,Cheeger69}]
 \label{thm:Poincare}For any logconcave density p in $\R^{n}$
and any function $g$ in $\Rn$, we have
\[
\Var_{p}\left(g(x)\right)\lesssim\psi_{p}^{-2}\cdot\E_{p}\left(\norm{\nabla g(x)}_{2}^{2}\right)
\]
\end{thm}

\section{Main proof}

The goal of this section is to prove the following theorem, which
implies Theorem \ref{thm:logsob}.
\begin{thm}
\label{lem:boundAgivesKLS}Given an isotropic logconcave distribution
$p$ with support of diameter $D$. Then, for any measurable subset
$S$, 
\[
p(\partial S)\geq\Omega\left(\frac{\log\frac{1}{p(S)}}{D}+\sqrt{\frac{\log\frac{1}{p(S)}}{D}}\right)p(S).
\]
\end{thm}

\subsection{Bounding the spectral norm of the covariance matrix}

The main lemma of this section is the following. 
\begin{lem}
\label{lem:norm_At}Assume that for $k\geq1$, $\psi_{p}\gtrsim\left(\tr A^{k}\right)^{-1/2k}$
for all logconcave distribution $p$ with covariance $A$. There is
some universal constant $c\geq0$ such that for any $0\leq T\leq\frac{1}{c\cdot kn^{1/k}}$,
we have that
\[
\P(\max_{t\in[0,T]}\norm{A_{t}}_{\spe}\geq2)\leq2\exp(-\frac{1}{cT}).
\]
\end{lem}

\begin{rem*}
Recall that Theorem \ref{thm:LeeV} showed that the assumption for
$k=2$.
\end{rem*}
We defer the proof of the following lemma to the end of this section. 
\begin{lem}
\label{lem:stoc_du}We have that
\[
du(A_{t})=\alpha_{t}^{T}dW_{t}+\beta_{t}dt
\]
where
\begin{align*}
\alpha_{t}= & \frac{1}{\kappa_{t}}\E_{x\sim\tilde{p}_{t}}x^{T}(uI-A_{t})^{-(q+1)}x\cdot x,\\
\frac{\beta_{t}}{q+1}\leq & \frac{1}{2\kappa_{t}}\E_{x,y\sim\tilde{p}_{t}}x^{T}(uI-A_{t})^{-1}y\cdot x^{T}(uI-A_{t})^{-(q+1)}y\cdot x^{T}y\\
 & -\frac{1}{\kappa_{t}^{2}}\E_{x,y\sim\tilde{p}_{t}}x^{T}(uI-A_{t})^{-(q+1)}x\cdot y^{T}(uI-A_{t})^{-(q+2)}y\cdot x^{T}y\\
 & +\frac{1}{2\kappa_{t}^{3}}\tr((uI-A_{t})^{-(q+2)})\cdot\E_{x,y\sim\tilde{p}_{t}}x^{T}(uI-A_{t})^{-(q+1)}x\cdot y^{T}(uI-A_{t})^{-(q+1)}y\cdot x^{T}y,\\
\kappa_{t}= & \tr((uI-A_{t})^{-(q+1)})
\end{align*}
and $\tilde{p}_{t}$ be the translation of $p_{t}$ defined by $\tilde{p}_{t}(x)=p_{t}(x+\mu_{t})$.
\end{lem}

To estimate $\alpha_{t}$, we need the following lemma proved in \cite{LeeV17KLS}.
\begin{lem}[{\cite[Lemma 25]{LeeV17KLS}}]
\label{lem:norm_expectation}Given a logconcave distribution $p$
with mean $\mu$ and covariance $A$. For any $C\succeq0$, we have
that
\[
\norm{\E_{x\sim p}(x-\mu)^{T}C(x-\mu)(x-\mu)}_{2}=O(\norm A_{\spe}^{1/2})\tr\left(A^{1/2}CA^{1/2}\right).
\]
\end{lem}

To estimate $\beta_{t}$, we prove the following bound that crucially
uses the KLS bound for non-isotropic logconcave distribution (the
assumption in Lemma \ref{lem:norm_At}).
\begin{lem}
\label{lem:tensor_bound}Under the assumption as Lemma \ref{lem:norm_At}.
Given a logconcave distribution $p$ with mean $\mu$ and covariance
$A$. For any $B^{(1)},B^{(2)},B^{(3)}\succeq0$, 
\begin{align*}
 & \left|\E_{x,y\sim p}(x-\mu)^{T}B^{(1)}(y-\mu)\cdot(x-\mu)^{T}B^{(2)}(y-\mu)\cdot(x-\mu)^{T}B^{(3)}(y-\mu)\right|\\
\leq & O(1)\cdot\tr(A^{\frac{1}{2}}B^{(1)}A^{\frac{1}{2}})\cdot\left(\tr(A^{\frac{1}{2}}B^{(2)}A^{\frac{1}{2}})^{k}\right)^{1/k}\cdot\norm{A^{\frac{1}{2}}B^{(3)}A^{\frac{1}{2}}}_{\spe}.
\end{align*}
\end{lem}

\begin{proof}
Without loss of generality, we can assume $p$ is isotropic. Furthermore,
we can assume $B^{(1)}$ is diagonal. Let $\Delta^{(k)}=\E_{x\sim p}xx^{T}\cdot x^{T}e_{k}$.
Then, we have that 
\begin{align}
\E_{x,y\sim p}x^{T}B^{(1)}y\cdot x^{T}B^{(2)}y\cdot x^{T}B^{(3)}y & =\sum_{k}B_{kk}^{(1)}\tr(B^{(2)}\Delta^{(k)}B^{(3)}\Delta^{(k)})\nonumber \\
 & \leq\norm{B^{(3)}}_{\spe}\sum_{k}B_{kk}^{(1)}\tr(\Delta^{(k)}B^{(2)}\Delta^{(k)}).\label{eq:tensor_bound}
\end{align}
Note that
\begin{align*}
\tr(\Delta^{(k)}B^{(2)}\Delta^{(k)}) & =\E_{x\sim p}x^{T}B^{(2)}\Delta^{(k)}x\cdot x_{k}\\
 & \leq\sqrt{\E x_{k}^{2}}\sqrt{\Var_{x\sim p}x^{T}B^{(2)}\Delta^{(k)}x}\\
 & =\sqrt{\Var_{y\sim\tilde{p}}y^{T}\left(B^{(2)}\right)^{\frac{1}{2}}\Delta^{(k)}\left(B^{(2)}\right)^{-\frac{1}{2}}y}
\end{align*}
where $\tilde{p}$ is the distribution given by $\left(B^{(2)}\right)^{\frac{1}{2}}x$
where $x\sim p$. Note that $\tilde{p}$ is a logconcave distribution
with mean $0$ and covariance $B^{(2)}$. Theorem \ref{thm:Poincare}
together with the assumption that $\psi_{p}\gtrsim\left(\tr A^{k}\right)^{-1/2k}$,
we have
\[
\Var_{y\sim\tilde{p}}y^{T}\left(B^{(2)}\right)^{\frac{1}{2}}\Delta^{(k)}\left(B^{(2)}\right)^{-\frac{1}{2}}y\apprle\left(\tr(B^{(2)})^{k}\right)^{1/2k}\cdot\E_{y\sim\tilde{p}}\norm{\left(B^{(2)}\right)^{\frac{1}{2}}\Delta^{(k)}\left(B^{(2)}\right)^{-\frac{1}{2}}y}^{2}
\]
Hence, we have that
\begin{align*}
\tr(\Delta^{(k)}B^{(2)}\Delta^{(k)}) & \leq\left(\tr(B^{(2)})^{k}\right)^{1/2k}\sqrt{\E_{y\sim\tilde{p}}\norm{\left(B^{(2)}\right)^{\frac{1}{2}}\Delta^{(k)}\left(B^{(2)}\right)^{-\frac{1}{2}}y}^{2}}\\
 & =\left(\tr(B^{(2)})^{k}\right)^{1/2k}\sqrt{\E_{y\sim\tilde{p}}\tr\left(B^{(2)}\right)^{-\frac{1}{2}}\Delta^{(k)}B^{(2)}\Delta^{(k)}\left(B^{(2)}\right)^{-\frac{1}{2}}yy^{T}}\\
 & =\left(\tr(B^{(2)})^{k}\right)^{1/2k}\sqrt{\tr\Delta^{(k)}B^{(2)}\Delta^{(k)}}.
\end{align*}
Hence, we have that
\[
\tr(\Delta^{(k)}B^{(2)}\Delta^{(k)})\leq\left(\tr(B^{(2)})^{k}\right)^{1/k}.
\]
Putting it into (\ref{eq:tensor_bound}) gives that
\[
\left|\E_{x,y\sim p}x^{T}B^{(1)}y\cdot x^{T}B^{(2)}y\cdot x^{T}B^{(3)}y\right|\apprle\tr B^{(1)}\cdot\left(\tr(B^{(2)})^{k}\right)^{1/k}\cdot\norm{B^{(3)}}_{\spe}.
\]
\end{proof}
\begin{lem}
\label{lem:bound_on_alpha_beta}Under the assumption as Lemma \ref{lem:norm_At}.
For $q=\left\lceil k\right\rceil $, let $u_{t}=u(A_{t})$, we have
that
\[
du_{t}=\alpha_{t}^{T}dW_{t}+\beta_{t}dt
\]
with
\[
\norm{\alpha_{t}}_{2}\apprle u_{t}^{\frac{3}{2}}\quad\text{and}\quad\beta_{t}\apprle ku_{t}^{3}\Phi^{1/k}.
\]
\end{lem}

\begin{proof}
For $\alpha_{t}$, we use Lemma \ref{lem:norm_expectation} and get
that
\begin{align*}
\norm{\E_{x\sim\tilde{p}_{t}}x^{T}(u_{t}I-A_{t})^{-(q+1)}x\cdot x}_{2} & \apprle\norm{A_{t}}_{\spe}^{1/2}\tr(A_{t}(u_{t}I-A_{t})^{-(q+1)})\\
 & \leq u_{t}^{\frac{3}{2}}\kappa_{t}.
\end{align*}

For $\beta_{t}$, we bound each term separately. For the first term,
Lemma \ref{lem:tensor_bound} shows that
\begin{align*}
 & \E_{x,y\sim\tilde{p}_{t}}x^{T}(u_{t}I-A_{t})^{-1}y\cdot x^{T}(u_{t}I-A_{t})^{-(q+1)}y\cdot x^{T}y\\
\apprle & \tr(A_{t}^{\frac{1}{2}}(u_{t}I-A_{t})^{-(q+1)}A_{t}^{\frac{1}{2}})\cdot\left(\tr(A_{t}^{\frac{1}{2}}(u_{t}I-A_{t})^{-1}A_{t}^{\frac{1}{2}})^{k}\right)^{1/k}\cdot\norm{A_{t}}_{\spe}\\
\leq & \norm{A_{t}}_{\spe}^{3}\cdot\kappa_{t}\cdot\left(\tr(u_{t}I-A_{t})^{-k}\right)^{1/k}\\
\leq & u_{t}^{3}\cdot\kappa_{t}\cdot\left(\tr(u_{t}I-A_{t})^{-k}\right)^{1/k}.
\end{align*}
For the second term, we use Lemma \ref{lem:norm_expectation} and
get that
\begin{align*}
 & \E_{x,y\sim\tilde{p}_{t}}x^{T}(u_{t}I-A_{t})^{-(q+1)}x\cdot y^{T}(u_{t}I-A_{t})^{-(q+2)}y\cdot x^{T}y\\
\leq & \norm{\E_{x\sim\tilde{p}_{t}}x^{T}(u_{t}I-A_{t})^{-(q+1)}x\cdot x}_{2}\norm{\E_{y\sim\tilde{p}_{t}}y^{T}(u_{t}I-A_{t})^{-(q+2)}y\cdot y}_{2}\\
\apprle & \norm{A_{t}}_{\spe}^{1/2}\tr(A_{t}(u_{t}I-A_{t})^{-(q+1)})\cdot\norm{A_{t}}_{\spe}^{1/2}\tr(A_{t}(u_{t}I-A_{t})^{-(q+2)})\\
\leq & u_{t}^{3}\cdot\kappa_{t}\cdot\tr((u_{t}I-A_{t})^{-(q+2)}).
\end{align*}
For the third term, the same calculation shows that 
\begin{align*}
 & \tr((u_{t}I-A_{t})^{-(q+2)})\cdot\E_{x,y\sim\tilde{p}_{t}}x^{T}(u_{t}I-A_{t})^{-(q+1)}x\cdot y^{T}(u_{t}I-A_{t})^{-(q+1)}y\cdot x^{T}y\\
\apprle & u_{t}^{3}\cdot\tr((u_{t}I-A_{t})^{-(q+2)})\cdot\kappa_{t}^{2}.
\end{align*}
Combining all three terms, we have 
\begin{align}
\beta_{t} & \apprle qu_{t}^{3}\left(\left(\tr(u_{t}I-A_{t})^{-k}\right)^{1/k}+\frac{\tr((u_{t}I-A_{t})^{-(q+2)})}{\kappa_{t}}\right).\label{eq:beta_t_bound}
\end{align}

Using that $q=\left\lceil k\right\rceil $, we have that 
\[
\beta_{t}\apprle ku_{t}^{3}\left(\Phi^{1/k}+\frac{\tr((u_{t}I-A_{t})^{-(\left\lceil k\right\rceil +2)})}{\tr((u_{t}I-A_{t})^{-(\left\lceil k\right\rceil +1)})}\right)\leq ku_{t}^{3}\left(\Phi^{1/k}+\tr((u_{t}I-A_{t})^{-(\left\lceil k\right\rceil +2)})^{\frac{1}{\left\lceil k\right\rceil +2}}\right)\apprle ku_{t}^{3}\Phi^{1/k}.
\]
\end{proof}
Now, we are ready to upper bound $\norm{A_{t}}_{\spe}$.
\begin{proof}[Proof of Lemma \ref{lem:norm_At}.]
 Consider the potential $\Psi_{t}=-(u_{t}+1)^{-2}$. Using Lemma
\ref{lem:bound_on_alpha_beta}, we have that
\begin{align*}
d\Psi_{t} & =2(u_{t}+1)^{-3}(\alpha_{t}^{T}dW_{t}+\beta_{t}dt)-3(u_{t}+1)^{-4}\norm{\alpha_{t}}^{2}dt\\
 & \defeq\gamma_{t}^{T}dW_{t}+\eta_{t}dt.
\end{align*}
Note that
\[
\norm{\gamma_{t}}_{2}^{2}=\norm{2(u_{t}+1)^{-3}\alpha_{t}}_{2}^{2}\leq O(1)(u_{t}+1)^{-6}u_{t}^{3}\leq c
\]
and
\[
\eta_{t}\leq2(u_{t}+1)^{-3}O(u_{t}^{3})k\Phi^{1/k}\leq ck\Phi^{1/k}
\]
for some universal constant $c$.

Let $Y_{t}$ be the process $dY_{t}=\gamma_{t}^{T}dW_{t}$. By Theorem
\ref{thm:Dubins}, there exists a Wiener process $\widetilde{W}_{t}$
such that $Y_{t}$ has the same distribution as $\widetilde{W}_{[Y]_{t}}$.
Using the reflection principle for 1-dimensional Brownian motion,
we have that
\[
\P(\max_{t\in[0,T]}Y_{t}\geq\gamma)\leq\P(\max_{t\in[0,cT]}\widetilde{W}_{t}\geq\gamma)=2\P(\widetilde{W}_{cT}\geq\gamma)\leq2\exp(-\frac{\gamma^{2}}{2cT}).
\]
Therefore, we have that
\[
\P(\max_{t\in[0,T]}\Psi_{t}-\Psi_{0}\geq ck\Phi^{1/k}T+\gamma)\leq2\exp(-\frac{\gamma^{2}}{2cT}).
\]
Set $\Phi=2^{-k}n$. At $t=0$, we have $\tr(u_{0}I-I)^{-k}=2^{-k}n$.
Therefore, $u_{0}=\frac{3}{2}$ and $\Psi_{0}=-\frac{4}{25}$. Using
the assumptions that $T\leq\frac{1}{25ck\Phi^{1/k}}$, we have that
\[
\P(\max_{t\in[0,T]}\left(-(u_{t}+1)^{-2}\right)\geq-\frac{3}{25}+\gamma)\leq2\exp(-\frac{\gamma^{2}}{2cT}).
\]
The result follows from setting $\gamma=\frac{1}{120}$.
\end{proof}

\subsection{Calculus with the BSS potential}

Here we prove Lemma \ref{lem:stoc_du}.
\begin{lem}
\label{lem:derivative_u}We have that 
\[
Du(X)[H]=\frac{\tr((uI-X)^{-(q+1)}H)}{\tr((uI-X)^{-(q+1)})}
\]
and
\begin{align*}
D^{2}u(X)[H_{1},H_{2}]= & \frac{\sum_{k=1}^{q+1}\tr((uI-X)^{-k}H_{1}(uI-X)^{-(q+2-k)}H_{2})}{\tr((uI-X)^{-(q+1)})}\\
 & -(q+1)\frac{\tr((uI-X)^{-(q+1)}H_{1})\tr((uI-X)^{-(q+2)}H_{2})}{\tr((uI-X)^{-(q+1)})^{2}}\\
 & -(q+1)\frac{\tr((uI-X)^{-(q+1)}H_{2})\tr((uI-X)^{-(q+2)}H_{1})}{\tr((uI-X)^{-(q+1)})^{2}}\\
 & +(q+1)\frac{\tr((uI-X)^{-(q+1)}H_{1})\tr((uI-X)^{-(q+1)}H_{2})\tr((uI-X)^{-(q+2)})}{\tr((uI-X)^{-(q+1)})^{3}}.
\end{align*}
\end{lem}

\begin{proof}
Taking derivative of the equation (\ref{eq:def_u}), we have 
\[
-q\tr((uI-X)^{-(q+1)})\cdot Du(X)[H]+q\tr((uI-X)^{-(q+1)}H)=0.
\]
Therefore,
\[
Du(X)[H]=\frac{\tr((uI-X)^{-(q+1)}H)}{\tr((uI-X)^{-(q+1)})}.
\]
Taking derivative again on both sides,
\begin{align*}
D^{2}u(X)[H_{1},H_{2}]= & \frac{\sum_{k=1}^{q+1}\tr((uI-X)^{-k}H_{1}(uI-X)^{-(q+2-k)}H_{2})}{\tr((uI-X)^{-(q+1)})}\\
 & -\frac{(q+1)\tr((uI-X)^{-(q+1)}H_{1})\tr((uI-X)^{-(q+2)}H_{2})}{\tr((uI-X)^{-(q+1)})^{2}}\\
 & -(q+1)\frac{\tr((uI-X)^{-(q+2)}H_{1})}{\tr((uI-X)^{-(q+1)})}Du(X)[H_{2}]\\
 & +(q+1)\frac{\tr((uI-X)^{-(q+1)}H_{1})\tr((uI-X)^{-(q+2)})}{\tr((uI-X)^{-(q+1)})^{2}}Du(X)[H_{2}].
\end{align*}
Substituting the formula of $Du(X)[H_{2}]$ and organizing the term,
we have the result.
\end{proof}
To simplify the first term in the Hessian, we need the following Lemma:
\begin{lem}[\cite{Eldan2013}]
\label{lem:tensor_shift}For any positive definite matrix $A$, symmetric
matrix $\Delta$ and any $\alpha,\beta\geq0$, we have that
\[
\tr(A^{\alpha}\Delta A^{\beta}\Delta)\leq\tr(A^{\alpha+\beta}\Delta^{2}).
\]
\end{lem}

\begin{proof}
Without loss of generality, we can assume $A$ is diagonal by rotating
the space. Hence, we have that 
\begin{align*}
\tr(A^{\alpha}\Delta A^{\beta}\Delta) & =\sum_{i,j}A_{ii}^{\alpha}A_{jj}^{\beta}\Delta_{ij}^{2}\\
 & \leq\sum_{i,j}\left(\frac{\alpha}{\alpha+\beta}A_{ii}^{\alpha+\beta}+\frac{\beta}{\alpha+\beta}A_{jj}^{\alpha+\beta}\right)\Delta_{ij}^{2}\\
 & =\frac{\alpha}{\alpha+\beta}\sum_{i,j}A_{ii}^{\alpha+\beta}\Delta_{ij}^{2}+\frac{\beta}{\alpha+\beta}\sum_{i,j}A_{jj}^{\alpha+\beta}\Delta_{ij}^{2}\\
 & =\tr(A^{\alpha+\beta}\Delta^{2}).
\end{align*}
\end{proof}
Now, we are already to upper bound $du(A_{t})$.
\begin{proof}[Proof of Lemma \ref{lem:stoc_du}]
Using Lemma \ref{lem:derivative_u} and Itô's formula, we have that
\begin{align*}
du(A_{t})= & \frac{\tr((uI-A_{t})^{-(q+1)}dA_{t})}{\tr((uI-A_{t})^{-(q+1)})}\\
 & +\frac{1}{2}\sum_{ijkl}\frac{\sum_{k=1}^{q+1}\tr((uI-A_{t})^{-k}e_{ij}(uI-A_{t})^{-(q+2-k)}e_{kl})}{\tr((uI-A_{t})^{-(q+1)})}d[A_{ij},A_{kl}]_{t}\\
 & -\frac{q+1}{2}\sum_{ijkl}\frac{\tr((uI-A_{t})^{-(q+1)}e_{ij})\tr((uI-A_{t})^{-(q+2)}e_{kl})}{\tr((uI-A_{t})^{-(q+1)})^{2}}d[A_{ij},A_{kl}]_{t}\\
 & -\frac{q+1}{2}\sum_{ijkl}\frac{\tr((uI-A_{t})^{-(q+1)}e_{kl})\tr((uI-A_{t})^{-(q+2)}e_{ij})}{\tr((uI-A_{t})^{-(q+1)})^{2}}d[A_{ij},A_{kl}]_{t}\\
 & +\frac{q+1}{2}\sum_{ijkl}\frac{\tr((uI-A_{t})^{-(q+1)}e_{ij})\tr((uI-A_{t})^{-(q+1)}e_{kl})\tr((uI-A_{t})^{-(q+2)})}{\tr((uI-A_{t})^{-(q+1)})^{3}}d[A_{ij},A_{kl}]_{t}.
\end{align*}

For brevity, we let $\tilde{p}_{t}$ be the translation of $p_{t}$
that has mean $0$, i.e. $\tilde{p}_{t}(x)=p_{t}(x+\mu_{t})$. Let
$\Delta^{(z)}=\E_{x\sim\tilde{p}_{t}}xx^{T}x_{z}$. Then, Lemma \ref{lem:dp_dA}
shows that $dA_{t}=\sum_{z}\Delta_{z}dW_{t,z}-A_{t}^{2}dt$ where
$W_{t,z}$ is the $z^{th}$ coordinate of $W_{t}$. Therefore, 
\begin{equation}
d[A_{ij},A_{kl}]_{t}=\sum_{z}\Delta_{ij}^{(z)}\Delta_{kl}^{(z)}dt.\label{eq:dA_bracket}
\end{equation}
Using the formula for $dA_{t}$ (\ref{eq:dA}) and $d[A_{ij},A_{kl}]_{t}$
(\ref{eq:dA_bracket}), we have that
\begin{align*}
du(A_{t})= & \frac{\tr\left((uI-A_{t})^{-(q+1)}\left(\E_{x\sim\tilde{p}_{t}}xx^{T}x^{T}dW_{t}-A_{t}^{2}dt\right)\right)}{\tr((uI-A_{t})^{-(q+1)})}\\
 & +\frac{1}{2}\sum_{z}\frac{\sum_{k=1}^{q+1}\tr((uI-A_{t})^{-k}\Delta^{(z)}(uI-A_{t})^{-(q+2-k)}\Delta^{(z)})}{\tr((uI-A_{t})^{-(q+1)})}dt\\
 & -(q+1)\sum_{z}\frac{\tr((uI-A_{t})^{-(q+1)}\Delta^{(z)})\tr((uI-A_{t})^{-(q+2)}\Delta^{(z)})}{\tr((uI-A_{t})^{-(q+1)})^{2}}dt\\
 & +\frac{q+1}{2}\sum_{z}\frac{\tr((uI-A_{t})^{-(q+1)}\Delta^{(z)})\tr((uI-A_{t})^{-(q+1)}\Delta^{(z)})\tr((uI-A_{t})^{-(q+2)})}{\tr((uI-A_{t})^{-(q+1)})^{3}}dt.
\end{align*}
Using Lemma \ref{lem:tensor_shift}, 
\[
\tr((uI-A_{t})^{-k}\Delta^{(z)}(uI-A_{t})^{-(q+2-k)}\Delta^{(z)})\leq\tr((uI-A_{t})^{-1}\Delta^{(z)}(uI-A_{t})^{-(q+1)}\Delta^{(z)})
\]
for all $1\leq k\leq q+1$. 

Let $\kappa_{t}=\tr((uI-A_{t})^{-(q+1)})$, then
\begin{align*}
du(A_{t})\leq & \frac{1}{\kappa_{t}}\E_{x\sim\tilde{p}_{t}}x^{T}(uI-A_{t})^{-(q+1)}xx^{T}dW_{t}\\
 & +\frac{q+1}{2\kappa_{t}}\E_{x,y\sim\tilde{p}_{t}}\sum_{z}\tr((uI-A_{t})^{-1}xx^{T}x_{z}(uI-A_{t})^{-(q+1)}yy^{T}y_{z})dt\\
 & -\frac{q+1}{\kappa_{t}^{2}}\E_{x,y\sim\tilde{p}_{t}}\sum_{z}\tr((uI-A_{t})^{-(q+1)}xx^{T}x_{z})\tr((uI-A_{t})^{-(q+2)}yy^{T}y_{z})dt\\
 & +\frac{q+1}{2\kappa_{t}^{3}}\E_{x,y\sim\tilde{p}_{t}}\sum_{z}\tr((uI-A_{t})^{-(q+1)}xx^{T}x_{z})\tr((uI-A_{t})^{-(q+1)}yy^{T}y_{z})\tr((uI-A_{t})^{-(q+2)})dt.
\end{align*}
Rearranging the terms, we have the result.
\end{proof}

\subsection{Bounding the size of any initial set}

Fix any set $E\subset\Rn$ and define $g_{t}=p_{t}(E)$.
\begin{lem}
\label{lem:volume} The random variable $g_{t}$ is a martingale satisfying
\[
d[g_{t}]_{t}\leq D^{2}g_{t}^{2}dt
\]
and
\[
d[g_{t}]_{t}\leq30\norm{A_{t}}_{\spe}\cdot g_{t}^{2}\log^{2}\left(\frac{1}{g_{t}}\right)dt.
\]
\end{lem}

\begin{proof}
Note that
\begin{align*}
dg_{t} & =\left\langle \int_{E}(x-\mu_{t})p_{t}(x)dx,dW_{t}\right\rangle .
\end{align*}
Therefore, we have that
\[
d[g_{t}]_{t}=\norm{\int_{E}(x-\mu_{t})p_{t}(x)dx}_{2}^{2}dt.
\]
We bound the norm in two different ways. On one hand, we note that
\begin{equation}
\norm{\int_{E}(x-\mu_{t})p_{t}(x)dx}_{2}\leq\sup\norm{x-\mu_{t}}_{2}\left(\int_{E}p_{t}(x)dx\right)=\sup\norm{x-\mu_{t}}_{2}g_{t}\leq D\cdot g_{t}\label{eq:gt1}
\end{equation}
On the other hand, for any $k\geq1$, we have that 
\begin{align}
\norm{\int_{E}(x-\mu_{t})p_{t}(x)dx}_{2} & =\max_{\norm{\zeta}_{2}=1}\int_{E}(x-\mu_{t})^{T}\zeta\cdot p_{t}(x)dx\nonumber \\
 & \leq\max_{\norm{\zeta}_{2}=1}\left(\int_{E}\left|(x-\mu_{t})^{T}\zeta\right|^{k}\cdot p_{t}(x)dx\right)^{\frac{1}{k}}\left(\int_{E}p_{t}(x)dx\right)^{1-\frac{1}{k}}\nonumber \\
 & \leq2k\norm{A_{t}}_{\spe}^{1/2}\cdot g_{t}{}^{1-\frac{1}{k}}\label{eq:gt2}
\end{align}
where we used Lemma \ref{lem:lcmom} at the end. Setting $k=\log(\frac{1}{g_{t}})$,
we have the result.
\end{proof}
Using this, we can bound how fast $\log\frac{1}{g_{t}}$ changes.
\begin{lem}
\label{lem:set_large} For any $T\geq0$ and $\gamma\geq0$, we have
that
\[
\P\left(\text{ for all }0\leq t\leq T:\,\log\frac{1}{g_{0}}+\frac{1}{2}D^{2}t+\gamma\geq\log\frac{1}{g_{t}}\geq\log\frac{1}{g_{0}}-\gamma\right)\geq1-4\exp(-\frac{\gamma^{2}}{2TD^{2}}).
\]
\end{lem}

\begin{proof}
Since $dg_{t}=g_{t}\alpha_{t}^{T}dW_{t}$ for some $\norm{\alpha_{t}}_{2}\leq D$
(from (\ref{eq:gt1}) in Lemma \ref{lem:volume}), using Itô's formula
(Lemma (\ref{lem:Ito})) we have that
\begin{align*}
d\log\frac{1}{g_{t}} & =-\frac{dg_{t}}{g_{t}}+\frac{1}{2}\frac{d[g_{t}]_{t}}{g_{t}^{2}}\\
 & =-\alpha_{t}^{T}dW_{t}+\frac{1}{2}\norm{\alpha_{t}}^{2}dt.
\end{align*}
Let $Y_{t}$ be the process $dY_{t}=\alpha_{t}^{T}dW_{t}$. By Theorem
\ref{thm:Dubins} and the reflection principle, we have that
\[
\P(\max_{t\in[0,T]}\left|Y_{t}\right|\geq\gamma)\leq4\exp(-\frac{\gamma^{2}}{2TD^{2}}).
\]
\end{proof}
Now, we bound $\E g_{t}\sqrt{\log\frac{1}{g_{t}}}$. This is the main
result of this section.
\begin{lem}
\label{lem:g_sqrt_g}Assume that $n\geq10$. There is some universal
constant $c\geq0$ such that for any measurable subset $E$ such that
$p_{0}(E)\leq c$ and any $T$ such that 
\[
0\leq T\leq c\cdot\max\left(\frac{1}{D^{2}}\log\frac{1}{p_{0}(E)},\frac{1}{\log\frac{1}{p_{0}(E)}+D}\right),
\]
we have that
\[
\E\left(p_{T}(E)\sqrt{\log\frac{1}{p_{T}(E)}}1_{p_{T}(E)\leq\frac{1}{2}}\right)\geq\frac{1}{5}p_{0}(E)\sqrt{\log\frac{1}{p_{0}(E)}}.
\]
\end{lem}

\begin{proof}
Fix any set $E\subset\Rn$ and define $g_{t}=p_{t}(E)$.

Case 1) $T\leq\frac{1}{8}D^{-2}\log\frac{1}{g_{0}}$. Lemma \ref{lem:set_large}
shows that
\[
\P(\log\frac{1}{g_{t}}\geq\frac{1}{4}\log\frac{1}{g_{0}}\text{ for all }0\leq t\leq T)\geq1-4g_{0}^{2}.
\]
Using $g_{0}\leq\frac{1}{16}$, we have that
\[
\E\left(g_{T}\sqrt{\log\frac{1}{g_{T}}}1_{g_{T}\leq\frac{1}{2}}\right)\geq\E\left(g_{T}\sqrt{\log\frac{1}{g_{T}}}1_{\log\frac{1}{g_{T}}\geq\frac{1}{4}\log\frac{1}{g_{0}}}\right)\geq\E\left(g_{T}1_{\log\frac{1}{g_{T}}\geq\frac{1}{4}\log\frac{1}{g_{0}}}\right)\sqrt{\frac{1}{4}\log\frac{1}{g_{0}}}.
\]
Since $\E g_{T}=g_{0}$ and $g_{T}\leq1$, we have that 
\[
\E\left(g_{T}1_{\log\frac{1}{g_{T}}\geq\frac{1}{4}\log\frac{1}{g_{0}}}\right)=g_{0}-\E\left(g_{T}1_{\log\frac{1}{g_{T}}<\frac{1}{4}\log\frac{1}{g_{0}}}\right)\geq g_{0}-4g_{0}^{2}\geq\frac{1}{2}g_{0}.
\]
Therefore, 
\[
\E\left(g_{T}\sqrt{\log\frac{1}{g_{T}}}1_{g_{T}\leq\frac{1}{2}}\right)\geq\frac{1}{4}g_{0}\sqrt{\log\frac{1}{g_{0}}}.
\]

Case 2) $T\geq\frac{1}{8}D^{-2}\log\frac{1}{g_{0}}$. Now, we assume
that $T\leq\frac{1}{2c(D+\log\frac{1}{g_{0}})}$ where $c\geq1$ is
the universal constant appears in Lemma \ref{lem:norm_At}. Note that
\begin{align*}
dg_{t}\sqrt{\log\frac{e}{g_{t}}} & =\frac{2\log\frac{e}{g_{t}}-1}{2\sqrt{\log\frac{e}{g_{t}}}}dg_{t}-\frac{2\log\frac{e}{g_{t}}+1}{8g_{t}\log^{\frac{3}{2}}\frac{e}{g_{t}}}d[g_{t}]_{t}.
\end{align*}
Since $dg_{t}=g_{t}\log\frac{e}{g_{t}}\alpha_{t}^{T}dW_{t}$ for some
$\norm{\alpha_{t}}_{2}\leq\sqrt{30}\norm{A_{t}}_{\spe}^{1/2}$ (from
(\ref{eq:gt2}) in Lemma \ref{lem:volume}), 
\begin{align*}
dg_{t}\sqrt{\log\frac{e}{g_{t}}} & =\frac{1}{2}g_{t}\sqrt{\log\frac{e}{g_{t}}}(2\log\frac{e}{g_{t}}-1)\alpha_{t}^{T}dW_{t}-\frac{1}{8}g_{t}\sqrt{\log\frac{e}{g_{t}}}(2\log\frac{e}{g_{t}}+1)\norm{\alpha_{t}}_{2}^{2}dt.
\end{align*}
For any $s\geq s'\geq0$, we have that
\begin{align}
\E g_{s}\sqrt{\log\frac{e}{g_{s}}} & =g_{s'}\sqrt{\log\frac{e}{g_{s'}}}-\frac{1}{8}\int_{s'}^{s}\E\left(g_{t}\sqrt{\log\frac{e}{g_{t}}}(2\log\frac{e}{g_{t}}+1)\norm{\alpha_{t}}_{2}^{2}\right)dt\nonumber \\
 & \geq g_{s'}\sqrt{\log\frac{e}{g_{s'}}}-12\int_{s'}^{s}\E\left(\norm{A_{t}}_{\spe}g_{t}\log^{\frac{3}{2}}\frac{e}{g_{t}}\right)dt.\label{eq:glogg1}
\end{align}

Using $T\geq\frac{1}{8}D^{-2}\log\frac{1}{g_{0}}$, Lemma \ref{lem:set_large}
shows that
\begin{equation}
\P(15D^{2}T\geq\max_{0\leq t\leq T}\log\frac{1}{g_{t}})\geq1-4g_{0}^{2}.\label{eq:glogg2}
\end{equation}
Now, using $T\leq\frac{1}{2c(\sqrt{n}+\log\frac{1}{g_{0}})}$, Lemma
\ref{lem:norm_At} (with $k=2$) shows that
\begin{equation}
\P(\max_{t\in[0,T]}\norm{A_{t}}_{\spe}\geq2)\leq2g_{0}^{2}.\label{eq:glogg3}
\end{equation}
Let $E$ be the event that both $\max_{0\leq t\leq T}\log\frac{1}{g_{t}}\leq15D^{2}T$
and $\max_{t\in[0,T]}\norm{A_{t}}_{\spe}\leq2$. Then, combining (\ref{eq:glogg2})
and (\ref{eq:glogg3}), we have that
\begin{align*}
\E\left(\norm{A_{t}}_{\spe}g_{t}\log^{\frac{3}{2}}\frac{e}{g_{t}}\right) & \leq2(1+15D^{2}T)\E\left(g_{t}\log^{\frac{1}{2}}\frac{e}{g_{t}}1_{E}\right)+\E\left(\norm{A_{t}}g_{t}\log^{\frac{3}{2}}\frac{e}{g_{t}}1_{E^{c}}\right)\\
 & \leq2(1+15D^{2}T)\E\left(g_{t}\log^{\frac{1}{2}}\frac{e}{g_{t}}\right)+12\sqrt{D}g_{0}^{2}.
\end{align*}
where we used that $\norm{A_{t}}_{\spe}\leq\sqrt{D}$ a.s. and $g_{t}\log^{\frac{3}{2}}\frac{e}{g_{t}}\leq2$
and $\P(E^{c})\leq6g_{0}^{2}$. Putting it into (\ref{eq:glogg1}),
for any $s'\leq s\leq T$, we have that 
\begin{align*}
\E g_{s}\sqrt{\log\frac{e}{g_{s}}} & \geq g_{s'}\sqrt{\log\frac{e}{g_{s'}}}-24(1+15D^{2}T)\int_{s'}^{s}\E g_{t}\sqrt{\log\frac{e}{g_{t}}}dt-144\sqrt{D}g_{0}^{2}T.
\end{align*}
Let $s^{*}$ be the $s\in[0,T]$ that maximizes $\E g_{s}\sqrt{\log\frac{e}{g_{s}}}$
and $f_{t}=\E g_{t}\sqrt{\log\frac{e}{g_{t}}}$. For all $T\geq s\geq s^{*}$,
we have that
\begin{align*}
f_{s} & \geq f_{s^{*}}-24(1+15D^{2}T)\int_{s'}^{s}f_{t}dt-144\sqrt{D}g_{0}^{2}T\\
 & \geq f_{s^{*}}-24(T+15D^{2}T^{2})f_{s^{*}}-144\sqrt{D}g_{0}^{2}T\\
 & \geq\frac{4}{5}f_{s^{*}}\geq\frac{4}{5}f_{0}
\end{align*}
where we used that $T\leq\frac{1}{10^{5}D}\leq\frac{1}{10^{5}}$ at
the end. Therefore, we have
\[
\E g_{T}\sqrt{\log\frac{e}{g_{T}}}\geq\frac{4}{5}g_{0}\sqrt{\log\frac{e}{g_{0}}}.
\]
Now, we note that
\begin{align*}
\E\left(g_{T}\sqrt{\log\frac{1}{g_{T}}}1_{g_{T}\leq\frac{1}{2}}\right) & \geq\frac{1}{2}\E\left(g_{T}\sqrt{\log\frac{e}{g_{T}}}1_{g_{T}\leq\frac{1}{2}}\right)\geq\frac{2}{5}g_{0}\sqrt{\log\frac{e}{g_{0}}}-\frac{\sqrt{\log2e}}{2}g_{0}\\
 & \geq\frac{1}{5}g_{0}\sqrt{\log\frac{e}{g_{0}}}
\end{align*}
where we used $g_{0}\leq\frac{1}{e^{10}}$ at the end.
\end{proof}

\subsection{Gaussian Case}

The next theorem can be found in \cite[Thm 1.1]{Ledoux1999}. We give
another proof for completeness.
\begin{thm}
\label{thm:Gaussian-iso}Let $h(x)=f(x)e^{-\frac{t}{2}\norm x^{2}}/\int f(y)e^{-\frac{t}{2}\norm y^{2}}dy$
where $f:\R^{n}\rightarrow\R_{+}$ is an integrable logconcave function
and $t\geq0$. Let $p(S)=\int_{S}h(x)dx$. For any $p(S)\leq\frac{1}{2}$,
we have
\[
p(\partial S)=\Omega\left(\sqrt{t}\right)\cdot p(S)\sqrt{\log\frac{1}{p(S)}}.
\]
\end{thm}

\begin{proof}
Let $g=p(S)$. Then, the desired statement can be written as 
\[
\int_{S}h(x)\:dx=g\int_{\R^{n}}h(x)\implies\int_{\partial S}h(x)\:dx\ge c\sqrt{t}g\sqrt{\log(\frac{1}{g})}\int_{\R^{n}}h(x)
\]
for some constant $c$. By the localization lemma \cite{KLS95}, if
there is a counterexample, there is a counterexample in one-dimension
where $h$ is of the form $h(x)=Ce^{-\gamma x-t\frac{x^{2}}{2}}$
restricted to some interval on the real line. Without loss of generality,
we can assume that $S$ is a single interval, otherwise, any interval
has smaller boundary measure and smaller volume. By rescaling, flipping
and shifting the function $h$, we can assume that $t=1$ and that
$h(x)=e^{-\frac{1}{2}x^{2}}1_{[a,b]}$ and that $S=[y,b]$ for some
$a<y<b$. 

It remains to show that for $g\le\frac{1}{2}$, 
\[
\frac{\int_{y}^{b}e^{-\frac{x^{2}}{2}}\,dx}{\int_{a}^{b}e^{-\frac{x^{2}}{2}}\:dx}=g\implies\frac{e^{-\frac{y^{2}}{2}}}{\int_{a}^{b}e^{-\frac{x^{2}}{2}}\,dx}\apprge g\sqrt{\log\frac{1}{g}}.
\]

There are three cases: $y\geq1$, $y\leq-1$ and $-1\leq y\leq-1$.
Let $A=\int_{a}^{b}e^{-\frac{x^{2}}{2}}\:dx$. In the first case $y\geq1$,
we note that the integral 
\begin{equation}
g\cdot A=\int_{y}^{b}e^{-\frac{x^{2}}{2}}\,dx\leq\int_{y}^{\infty}e^{-\frac{x^{2}}{2}}\,dx\apprle e^{-\frac{y^{2}}{2}}/y.\label{eq:gaussian_case_1}
\end{equation}
Rearrange the terms, we have that
\[
\frac{e^{-\frac{y^{2}}{2}}}{\int_{a}^{b}e^{-\frac{x^{2}}{2}}\,dx}\apprge g\cdot y.
\]
Using (\ref{eq:gaussian_case_1}), we have that $y\apprge\sqrt{\log\frac{1}{g\cdot A}}$
and hence
\[
\frac{e^{-\frac{y^{2}}{2}}}{\int_{a}^{b}e^{-\frac{x^{2}}{2}}\,dx}\apprge g\cdot\sqrt{\log\frac{1}{g\cdot A}}\apprge g\cdot\sqrt{\log\frac{1}{g}}
\]
where we used that $A=\int_{a}^{b}e^{-\frac{x^{2}}{2}}\:dx\leq\sqrt{2\pi}$
and $g\leq\frac{1}{2}$.

For the second case $y\leq-1$, since $g\leq\frac{1}{2}$, we have
that
\[
\frac{e^{-\frac{y^{2}}{2}}}{\int_{a}^{b}e^{-\frac{x^{2}}{2}}\,dx}\geq\frac{2e^{-\frac{y^{2}}{2}}}{\int_{a}^{y}e^{-\frac{x^{2}}{2}}\,dx}\geq\frac{e^{-\frac{y^{2}}{2}}}{\int_{-\infty}^{y}e^{-\frac{x^{2}}{2}}\,dx}\apprge\left|y\right|\geq1
\]
where we used that $\int_{-\infty}^{y}e^{-\frac{x^{2}}{2}}\,dx=\frac{O(1)}{\left|y\right|}e^{-\frac{y^{2}}{2}}$
at the end. This proves the second case because $g\sqrt{\log\frac{1}{g}}\apprle1$.

For the last case $|y|\leq1$, we note that
\[
\frac{e^{-\frac{y^{2}}{2}}}{\int_{a}^{b}e^{-\frac{x^{2}}{2}}\,dx}\geq\frac{e^{-\frac{1}{2}}}{\sqrt{2\pi}}\apprge1.
\]
\end{proof}

\subsection{Proof of main theorem}

We can now prove a bound on the isoperimetric constant. 
\begin{proof}[Proof of Theorem. \ref{lem:boundAgivesKLS}.]
 By Lemma \ref{lem:dp_dA}, $p_{t}$ is a martingale and therefore
\[
p(\partial E)=p_{0}(\partial E)=\E p_{T}(\partial E).
\]
Next, by the definition of $p_{T}$ (\ref{eq:dBt}), we have that
$p_{T}(x)\propto e^{c_{T}^{T}x-\frac{T}{2}\norm x^{2}}p(x)$ and Theorem
\ref{thm:Gaussian-iso} shows that if $p_{t}(E)\leq\frac{1}{2}$,
we have that
\[
p_{T}(\partial E)\apprge\sqrt{T}\cdot p_{t}(E)\sqrt{\log\frac{1}{p_{t}(E)}}.
\]
Hence, we have that
\[
p(\partial E)\apprge\sqrt{T}\cdot\E\left(p_{t}(E)\sqrt{\log\frac{1}{p_{t}(E)}}\cdot1_{p_{t}(E)\leq\frac{1}{2}}\right).
\]

Lemma \ref{lem:g_sqrt_g} shows that if 
\[
T\leq c\cdot\max\left(\frac{1}{D^{2}}\log\frac{1}{p_{0}(E)},\frac{1}{\log\frac{1}{p_{0}(E)}+D}\right)
\]
and $p_{0}(E)\leq c$ for some small enough constant $c$, we have
\begin{align*}
p(\partial E) & \apprge\sqrt{T}\cdot p_{0}(E)\sqrt{\log\frac{1}{p_{0}(E)}}\\
 & \apprge\left(\frac{\sqrt{\log\frac{1}{p_{0}(E)}}}{D}+\frac{1}{\sqrt{\log\frac{1}{p_{0}(E)}+D}}\right)p_{0}(E)\sqrt{\log\frac{1}{p_{0}(E)}}\\
 & \apprge\left(\frac{\log\frac{1}{p_{0}(E)}}{D}+\sqrt{\frac{\log\frac{1}{p_{0}(E)}}{\log\frac{1}{p_{0}(E)}+D}}\right)p_{0}(E)\\
 & \apprge\left(\frac{\log\frac{1}{p_{0}(E)}}{D}+\sqrt{\frac{\log\frac{1}{p_{0}(E)}}{D}}\right)p_{0}(E).
\end{align*}
If $p_{0}(E)\geq c$, the bound simply follows from Theorem \ref{thm:LeeV}.
\end{proof}

\section{Consequences of the improved isoperimetric inequality}

In this section, we give some consequences of the improved isoperimetric
inequality (Theorem \ref{lem:boundAgivesKLS}). First we note the
following Cheeger-type logarithmic isoperimetric inequality.
\begin{thm}[\cite{ledoux1994simple}]
Let $\mu$ be any absolutely continuous measure on $\Rn$ such that
for any open subsets $A$ of $\Rn$ with $\mu(A)\leq\frac{1}{2}$,
\[
\mu(\partial A)\geq\phi\cdot\mu(A)\sqrt{\log\frac{1}{\mu(A)}}.
\]
Then, for any $f$ such that $\int_{\Rn}f^{2}d\mu=1$, we have that
\[
\int_{\Rn}\left|\nabla f\right|^{2}d\mu\gtrsim\phi^{2}\cdot\int_{\Rn}f^{2}\log f^{2}d\mu.
\]
\end{thm}

Applying this and Theorem \ref{lem:boundAgivesKLS}, we have the following
result.
\begin{thm}
Given an isotropic logconcave distribution $p$ with diameter $D$.
For any $f$ such that $\int_{\Rn}f^{2}(x)p(x)dx=1$, we have that
\[
\int_{\Rn}\left|\nabla f(x)\right|^{2}p(x)dx\gtrsim\frac{1}{D}\cdot\int_{\Rn}f^{2}\log f^{2}p(x)dx.
\]
\end{thm}

\begin{rem*}
In general, isotropic logconcave distributions with diameter $D$
can have the coefficient as small as $O(\frac{1}{D})$, as we show
in Section \ref{sec:tight}. Therefore, our bound is tight up to constant.

\end{rem*}
Now we prove Theorem \ref{thm:conc}, an improved concentration inequality
for general Lipschitz functions and general isotropic logconcave densities.
\begin{proof}[Proof of Theorem \ref{thm:conc}.]
 We first prove the statement for the function $g(x)=\norm x$. Define
\[
E_{t}=\left\{ x\text{ such that }\norm x\geq\text{med}_{x\sim p}\norm x+t\right\} 
\]
 and $\alpha_{t}\defeq\log\frac{1}{p(E_{t})}$. By the definition
of median, we have that $\alpha_{0}=\log2$. We first give a weak
estimate on how fast $\alpha_{t}$ increases. Note that 
\[
\frac{d\alpha_{t}}{dt}=-\frac{1}{p(E_{t})}\frac{dp(E_{t})}{dt}=\frac{p(\partial E_{t})}{p(E_{t})}\geq\frac{c_{1}}{n^{1/4}}
\]
for some universal constant $c_{1}>0$ where we used that definition
of $p(\partial E_{t})$ to get $\frac{dp(E_{t})}{dt}=-p(\partial E_{t})$
and Theorem \ref{thm:LeeV} at the end. Therefore, 
\begin{equation}
\alpha_{t+s}\geq\alpha_{t}+c_{1}\cdot sn^{-1/4}\label{eq:weak_bound}
\end{equation}
for all $t,s\geq0$. To improve on this bound, we consider the distribution
$p_{t}$ defined by truncating the distribution $p$ to the set $\norm x\le\text{med}_{x\sim p}\norm x+t+c_{2}\sqrt{n}$
for some large enough constant $c_{2}$. By the estimate (\ref{eq:weak_bound}),
we see that $p(E_{t})$ only decreases by a tiny factor after truncation
and hence
\[
\frac{p(\partial E_{t})}{p(E_{t})}\geq\frac{1}{2}\frac{p_{t}(\partial E_{t})}{p_{t}(E_{t})}.
\]
Next, we note that $p_{t}$ is almost isotropic, namely its covariance
matrix $A_{t}$ satisfies $\frac{1}{2}I\preceq A_{t}\preceq2I$ (in
fact, it is exponentially close to $I$). Although Theorem \ref{lem:boundAgivesKLS}
only applies to the isotropic case, but we can always renormalize
the distribution $p_{t}$ to isotropic and then normalize it back.
Since the distribution $p_{t}$ is almost isotropic, such a re-normalization
does not change $p_{t}(\partial E)$ by more than a multiplicative
constant. 

Therefore, we can apply Theorem \ref{lem:boundAgivesKLS} and get
that
\begin{equation}
\frac{d\alpha_{t}}{dt}=\frac{p(\partial E_{t})}{p(E_{t})}\geq\frac{1}{2}\frac{p_{t}(\partial E_{t})}{p_{t}(E_{t})}\gtrsim\sqrt{\frac{\log\frac{1}{p_{t}(E_{t})}}{m+t}}\geq c_{3}\cdot\sqrt{\frac{\alpha_{t}}{m+t}}\label{eq:dalpha}
\end{equation}
for some universal constant $c_{3}>0$ where $m=\text{med}_{x\sim p}\norm x+c_{2}\sqrt{n}$.
Note that
\[
\frac{d\sqrt{\alpha_{t}}}{dt}\geq\frac{c_{3}}{2}\cdot\sqrt{\frac{1}{m+t}}
\]
and hence
\[
\sqrt{\alpha_{t}}-\sqrt{\alpha_{0}}\geq\frac{c_{3}}{2}\int_{0}^{t}\sqrt{\frac{1}{m+s}}ds=c_{3}\cdot\left(\sqrt{m+t}-\sqrt{m}\right).
\]

For $t\leq m$, we have that
\[
\sqrt{\alpha_{t}}-\sqrt{\alpha_{0}}\geq c_{3}\cdot\left(\sqrt{m}+\sqrt{m}\frac{t}{3m}-\sqrt{m}\right)\geq\frac{c_{3}\cdot t}{3\sqrt{m}}.
\]
and hence $\alpha_{t}\geq\frac{c_{3}^{2}}{9m}t^{2}.$

For $t\geq m$, we note that $\sqrt{1+x}-1\geq\frac{1}{3}\sqrt{x}$
for all $x\geq1$. Therefore, 
\[
\sqrt{m+t}-\sqrt{m}=\sqrt{m}\left(\sqrt{1+\frac{t}{m}}-1\right)\geq\frac{1}{3}\sqrt{m}\sqrt{\frac{t}{m}}=\frac{1}{3}\sqrt{t}.
\]
Hence, (\ref{eq:dalpha}) shows that $\sqrt{\alpha_{t}}-\sqrt{\alpha_{0}}\geq\frac{c_{3}}{3}\sqrt{t}.$

Combining both cases, we have that $\alpha_{t}\gtrsim\min\left(\frac{t^{2}}{\sqrt{n}},t\right)\gtrsim\frac{t^{2}}{t+\sqrt{n}}.$
Hence,
\begin{equation}
\P_{x\sim p}\left(\norm x-\text{med}_{x\sim p}\norm x\geq t\right)\leq\exp\left(-\Omega(1)\cdot\frac{t^{2}}{t+\sqrt{n}}\right).\label{eq:norm_x_small}
\end{equation}
By the same proof, we have that
\[
\P_{x\sim p}\left(\norm x-\text{med}_{x\sim p}\norm x\leq-t\right)\leq\exp\left(-\Omega(1)\cdot\frac{t^{2}}{t+\sqrt{n}}\right).
\]
This completes the proof for the Euclidean norm.For a general $1$-Lipschitz
function $g$, we define 
\[
E_{t}=\left\{ x\text{ such that }g(x)\geq m_{g}+t\right\} 
\]
where $m_{g}=\text{med}_{x\sim p}g(x)$. Again, we define $\alpha_{t}\defeq\log\frac{1}{p(E_{t})}$
and we have $\alpha_{0}=\log2$. To compute $\frac{p(\partial E_{t})}{p(E_{t})}$,
we consider the restriction of $p$ to a large ball. Let $\zeta\geq0$
to be chosen later and $B_{t,\zeta}$ be the ball centered at $0$
with radius $c_{4}\cdot(\sqrt{n}+\alpha_{t})+\zeta$ where $c_{4}$
is a constant. Let $p_{t,\zeta}$ be the distribution defined by 
\[
p_{t,\zeta}(A)=\frac{p(A\cap B_{t,\zeta})}{p(B_{t,\zeta})}.
\]
Choosing a large constant $c_{4}$ and using (\ref{eq:norm_x_small}),
we have that $p(B_{t,\zeta}^{c})\leq\frac{p(E_{t})}{100n}$ and that
$p_{t,\zeta}$ is almost isotropic. Since $p(B_{t,\zeta}^{c})$ is
so small, we have that $p_{t,\zeta}(E_{t})\geq\frac{1}{2}p(E_{t})$
and that
\[
p_{t,\zeta}(\partial E_{t})=\frac{p(\partial(E_{t}\cap B_{t,\zeta}))}{p(B_{t,\zeta})}\leq2p(\partial E_{t})+2p(\partial B_{t,\zeta}).
\]
Hence, 
\begin{equation}
\frac{p(\partial E_{t})}{p(E_{t})}\geq\frac{1}{4}\frac{p_{t,\zeta}(\partial E_{t})}{p_{t,\zeta}(E_{t})}-\frac{p_{t,\zeta}(\partial B_{t,\zeta})}{p(E_{t})}.\label{eq:pE_ratio}
\end{equation}
Since $p_{t,\zeta}$ is almost isotropic and has support of diameter
$O(\sqrt{n}+\alpha_{t}+\zeta)$, Theorem \ref{lem:boundAgivesKLS}
gives that 
\begin{equation}
\frac{p_{t,\zeta}(\partial E_{t})}{p_{t,\zeta}(E_{t})}\gtrsim\sqrt{\frac{\log\frac{1}{p_{t,\zeta}(E_{t})}}{\sqrt{n}+\alpha_{t}+\zeta}}\geq c_{5}\sqrt{\frac{\alpha_{t}}{\sqrt{n}+\alpha_{t}+\zeta}}\label{eq:ptE_ratio}
\end{equation}
for some universal constant $0<c_{5}<1$. To bound the second term
$p_{t,\zeta}(\partial B_{t,\zeta})$, we note that
\begin{align*}
\int_{0}^{1/c_{5}^{2}}p_{t,\zeta}(\partial B_{t,\zeta})d\zeta & \leq2\int_{0}^{1/c_{5}^{2}}p(\partial B_{t,\zeta})d\zeta\leq2p(B_{t,0}^{c})\leq\frac{p(E_{t})}{50n}
\end{align*}
Hence, there is $\zeta$ between $0$ and $1/c_{5}^{2}$ such that
$p_{t,\zeta}(\partial B_{t,\zeta})\leq\frac{p(E_{t})}{50n}c_{5}^{2}$.

Using this, (\ref{eq:pE_ratio}) and (\ref{eq:ptE_ratio}), we have
that
\begin{align}
\frac{p(\partial E_{t})}{p(E_{t})} & \geq\frac{c_{5}}{4}\sqrt{\frac{\alpha_{t}}{\sqrt{n}+\alpha_{t}+1/c_{5}^{2}}}-\frac{c_{5}^{2}}{50n}\nonumber \\
 & \geq\frac{c_{5}}{8}\sqrt{\frac{\alpha_{t}}{\sqrt{n}+\alpha_{t}+1/c_{5}^{2}}}+\frac{c_{5}}{16}\sqrt{\frac{1}{\sqrt{n}+1/c_{5}^{2}}}-\frac{c_{5}^{2}}{50n}\nonumber \\
 & \geq\frac{c_{5}}{8}\sqrt{\frac{\alpha_{t}}{\sqrt{n}+\alpha_{t}+1/c_{5}^{2}}}.\label{eq:p_ratio_est}
\end{align}

Next, we relate $\frac{p(\partial E_{t})}{p(E_{t})}$ with $dp(E_{t})$.
Since $g$ is $1$-Lipschitz, for any $x$ such that $\norm{x-y}_{2}\leq h$
and $y\in E_{t}$, we have that
\[
g(x)\geq m_{g}+t-h
\]
 Therefore, we have $B(E_{t},h)\subset E_{t-h}$ and, 
\begin{align*}
-\frac{dp(E_{t})}{dt} & =\lim_{h\rightarrow0}\frac{p\left(\left\{ g(x)\geq\text{med}_{x\sim p}g(x)+t-h\right\} \right)-p(E_{t})}{h}\\
 & \geq\lim_{h\rightarrow0}\frac{p\left(B(E_{t},h)\right)-p(E_{t})}{h}=p(\partial E_{t}).
\end{align*}
Using this with (\ref{eq:p_ratio_est}), we have that
\begin{align*}
\frac{d\alpha_{t}}{dt} & =-\frac{1}{p(E_{t})}\frac{dp(E_{t})}{dt}\geq\frac{p(\partial E_{t})}{p(E_{t})}\geq\frac{c_{5}}{8}\sqrt{\frac{\alpha_{t}}{\sqrt{n}+\alpha_{t}+1/c_{5}^{2}}}.
\end{align*}
Solving this equation, we again have that $\alpha_{t}\gtrsim\frac{t^{2}}{t+\sqrt{n}}.$
This proves that 
\[
\P_{x\sim p}\left(g(x)-\text{med}_{x\sim p}g(x)\geq t\right)\leq\exp\left(-\Omega(1)\cdot\frac{t^{2}}{t+\sqrt{n}}\right).
\]

The case of $g(x)-\text{med}_{x\sim p}g(x)\leq-t$ is the same by
taking the negative of $g$. To replace $\text{med}_{x\sim p}g(x)$
by $\E_{x\sim p}g(x)$, we simply use the concentration we just proved
to show that $\left|\E_{x\sim p}g(x)-\text{med}_{x\sim p}g(x)\right|\lesssim n^{\frac{1}{4}}$.
\end{proof}

\section{Small ball probability}

We use the following theorem that gives the small ball estimate for
logconcave distributions with a subgaussian tail, including strongly
logconcave distributions.
\begin{thm}[Theorem 1.3 in \cite{paouris2012small}]
\label{thm:smallpaouris}Given a logconcave distribution $p$ with
covariance matrix $A$. Suppose that $p$ has subgaussian constant
$b$, i.e.
\[
\E_{x\sim p}\left\langle x,\theta\right\rangle ^{2k}\leq\beta^{2k}\E_{g\sim\gamma}g^{2k}\text{ for all }\theta\in\Rn,k=1,2,\cdots
\]
where $\gamma$ is a standard Gaussian random variable. For any $y\in\Rn$
and $0\leq\varepsilon\leq c_{1}$, we have that
\[
\P_{x\sim p}\left(\norm{x-y}_{2}^{2}\leq\varepsilon\tr A\right)\leq\varepsilon^{c_{2}b^{-2}\norm A_{\spe}^{-1}\tr A}
\]
for some positive universal constant $c_{1}$ and $c_{2}$.
\end{thm}

We can now use stochastic localization to bound the small ball probability
(Theorem \ref{thm:small-ball}).
\begin{proof}[Proof of Theorem \ref{thm:small-ball}.]
Lemma \ref{lem:norm_At} shows that
\begin{equation}
\norm{A_{t}}_{\spe}\leq2\text{ for all }0\leq t\leq\frac{c}{kn^{1/k}}\label{eq:event_A}
\end{equation}
with probability at least $\frac{9}{10}$ for some universal constant
$c>0$.

Next, we bound $\tr A_{t}$. Lemma \ref{lem:dp_dA} shows that 
\[
d\tr A_{t}=\E_{x}\norm{x-\mu_{t}}^{2}(x-\mu_{t})^{T}dW_{t}-\tr A_{t}^{2}.
\]
Under the event \ref{eq:event_A}, we have that $\tr A_{t}^{2}\leq4n$.
Using Lemma \ref{lem:norm_expectation}, we have that
\[
\norm{\E_{x}\norm{x-\mu_{t}}^{2}(x-\mu_{t})}_{2}\lesssim n.
\]
Therefore, $d\tr A_{t}\geq O(n)dW_{t}-O(n)$. Hence, $\tr A_{t}$
is lower bounded by a Gaussian with mean $(1-c't)n$ and variance
$c''n^{2}t$ for some constants $c'$ and $c''$. Using this we have
that
\begin{equation}
\tr A_{t}\gtrsim n\text{ for all }0\leq t\leq c''\label{eq:event_B}
\end{equation}
with probability at least $\frac{8}{10}$ for some universal constant
$c''$.

Next, we let $E=\{x:\norm x\leq\varepsilon n\}$. We can assume the
diameter of $p$ is $2\sqrt{n}$ by truncating the domain, it does
not affect the small ball probability by more than a constant factor.
Let $g_{t}=p_{t}(E)$. Lemma \ref{lem:volume} shows that
\[
d[g_{t}]_{t}\leq(2\sqrt{n})^{2}g_{t}^{2}dt.
\]
Using this, we have that 
\[
d\log g_{t}=\eta_{t}dW_{t}-\frac{1}{2}\frac{1}{g_{t}^{2}}d[g_{t}]_{t}\geq\eta_{t}dW_{t}-(2\sqrt{n})^{2}dt
\]
with $0\leq\eta_{t}\leq2\sqrt{n}\log\frac{1}{\varepsilon}$. Solving
this equation, with probability at least $\frac{9}{10}$, we have
\begin{equation}
g_{t}\geq e^{-O(nt)}g_{0}.\label{eq:event_C}
\end{equation}

For the distribution $p_{t}$, we note that $p_{t}(x)=e^{-\frac{t}{2}\norm x_{2}^{2}}q_{t}(x)$
for some logconcave distribution $q_{t}$ (Definition \ref{def:A}).
Therefore, $p_{t}$ has subgaussian constant at most $\frac{1}{\sqrt{t}}$.
Theorem \ref{thm:smallpaouris} shows that
\[
\P_{x\sim p_{t}}\left(\norm{x-y}_{2}^{2}\leq\varepsilon\tr A_{t}\right)\leq\varepsilon^{\Omega(t\norm{A_{t}}_{\spe}^{-1}\tr A_{t})}.
\]
Therefore, under the events (\ref{eq:event_A}) and (\ref{eq:event_B}),
we have that 
\[
g_{t}\leq\exp(-\Omega(nt\log\frac{1}{\varepsilon})).
\]
Since the events (\ref{eq:event_A}), (\ref{eq:event_B}) and \ref{eq:event_C}
each happens with probability $\frac{9}{10}$, we have that with probability
at least $\frac{7}{10}$, for with $t=\Omega(k^{-1}n^{-1/k})$, 
\[
g_{0}\leq\exp(O(nt)-\Omega(nt\log\frac{1}{\varepsilon}))\leq\exp(-\Omega(nt\log\frac{1}{\varepsilon})).
\]
This completes the proof.
\end{proof}

\section{\label{sec:tight}Optimality of the bounds}
\begin{lem}
\label{lem:lower_bound_logsob}For any $\frac{n}{2}\geq D\geq2\sqrt{n}$,
there exists an isotropic logconcave distribution with diameter $D$
such that its log-Cheeger constant is $O(1/\sqrt{D})$ and its log-Sobolev
constant is $O(1/D)$. 
\end{lem}

In fact, we get a nearly lower tight bound on the mixing time of the
ball walk, from an arbitrary start, in terms of the number of proper
steps. Recall that by a proper step we mean steps where the current
point changes, not counting the steps that are discarded due to the
rejection probability. In our lower bound example, the local conductance,
i.e., the probability of a proper step, is at least a constant everywhere
and so the total number of steps in expectation is within a constant
factor of the number of proper steps.
\begin{lem}
\label{lem:lower_bound_ball}For any $\frac{n}{2}\geq D\geq2\sqrt{n}$,
there exists an isotropic convex body with diameter $D$ such that
ball walk mixes in $\widetilde{\Omega}(n^{2}D)$ proper steps.
\end{lem}

Both theorems are based on the following cone
\[
K=\left\{ x:\:0\le x_{1}\le n,\sum_{i=2}^{n}x_{i}^{2}\le\frac{1}{n}x_{1}^{2}\right\} ,
\]
and the truncated cone
\[
K_{D}=K\cap\left\{ x:\ (x_{1}-n)^{2}+\sum_{i=2}^{n}x_{i}^{2}\leq D^{2}\right\} .
\]
\begin{proof}[Proof of Lemma \ref{lem:lower_bound_logsob}.]
The convex body $K_{D}$ is nearly isotropic and has diameter $D$.
Let
\[
t_{0}=n-\sqrt{D^{2}-n}.
\]
Consider the subset $S=K\cap\left\{ t_{0}\le x_{1}\le t_{0}+1\right\} $.
Note that $S$ is fully contained in $K_{D}$ and that
\begin{align*}
p & =\frac{\vol(S)}{\vol(K_{D})}\approx\frac{\vol(S)}{\vol(K)}\\
 & =\left(\frac{t_{0}+1}{n}\right)^{n-1}-\left(\frac{t_{0}}{n}\right)^{n-1}\\
 & \approx e^{-D}.
\end{align*}
On the other hand, the expansion of $S$ is at most $2$. Therefore,
the log-Cheeger constant $\kappa$ of $K_{D}$ must satisfy 
\[
\kappa\sqrt{\log\frac{1}{p}}\le2
\]
or $\kappa=O\left(D^{-\frac{1}{2}}\right)$ as claimed. It is known
that the log-Sobolev constant $\rho=\Theta(\kappa^{2})$ (see e.g.,
\cite{ledoux1994simple}). This gives the second claim.
\end{proof}
The proof of the lower bound for the ball walk is more involved.
\begin{proof}[Proof of Lemma \ref{lem:lower_bound_ball}.]
Let the starting distribution be the uniform distribution in the
set $S=K_{D}\cap\left\{ t_{0}\le x_{1}\le t_{0}+1\right\} $. For
each point $x$, the local conductance is $\ell(x)=\frac{\vol(K_{D}\cap(x+\delta B))}{\vol(\delta B)}$,
the fraction of the $\delta$-ball around $x$ contained in $K_{D}$. 

The distribution at any time will remain spherically symmetric at
each time step. Each step along the $e_{1}$ direction is approximately
$\pm\frac{\delta}{\sqrt{n}}=\pm\frac{1}{n}$. An unbiased process
that moves $\pm\frac{1}{n}$ along $e_{1}$ in each step would take
$\Omega(n^{2}D^{2})$ steps to converge since the diameter is effectively
$nD$. But there is a slight drift in the positive direction towards
the base. This is because for points near the boundary of the cone
a step away from the base has higher rejection probability compared
to a step towards the base. We will now upper bound the drift and
therefore lower bound the number of steps needed.

More precisely, let $\delta=\frac{1}{\sqrt{n}}$ and $D<0.99n$. On
any slice $A(t)$, we can identify the bias felt by each point $y\in A(t)$.
It is the fraction of the $\delta$-ball around $y$ that is contained
in $K$ but its mirror image through the plane $x_{1}=y_{1}$ is not
contained in $K$. This fraction depends on the distance $d$ of $y$
to the boundary of $A(t)$. The points that feel a significant bias
are contained in the outer annulus of thickness $O\left(\frac{\delta}{\sqrt{n}}\right)=O\left(\frac{1}{n}\right)$.
To bound the fraction of the distribution in this annulus, we observe
that the density within each slice is spherically symmetric. In fact,
a stronger property holds. 
\begin{claim}
The distribution at any time in any cross-section is spherically symmetric
and unimodal. 
\end{claim}

Suppose the claim is true. Then, since the radius of any slice for
$t\ge t_{0}=n-\sqrt{D^{2}-n}$ is $\Omega(\sqrt{n})$, this is at
most $1-(1-O(1/n^{1.5}))^{n}=O(1/\sqrt{n})$ fraction of the slice.
Each point in the annulus has a probability of $O(1/\sqrt{n})$ of
making a move to the base with no symmetric move away. Thus, effectively,
the process can be viewed as a biased random walk on an interval of
length $D\ge C_{1}\sqrt{n}$ with the projection of each step along
$e_{1}$ is balanced with probability $1-\Omega\left(\frac{1}{\sqrt{n}}\right)$
and is $O\left(\frac{1}{n}\right)$ biased towards the base with probability
$O\left(\frac{1}{\sqrt{n}}\cdot\frac{1}{\sqrt{n}}\right)$. Thus the
drift is $+O\left(\frac{1}{n^{2}}\right)$ and it takes $\Omega\left(n^{2}D\right)$
steps to cover the $\Omega(D)$ distance to get within distance $1$
of the base along $e_{1}$. This is necessary to have total variation
distance less than (say) $0.1$ from the target speedy distribution
over the body. 
\end{proof}
\begin{figure}
\begin{centering}
\includegraphics[width=3in]{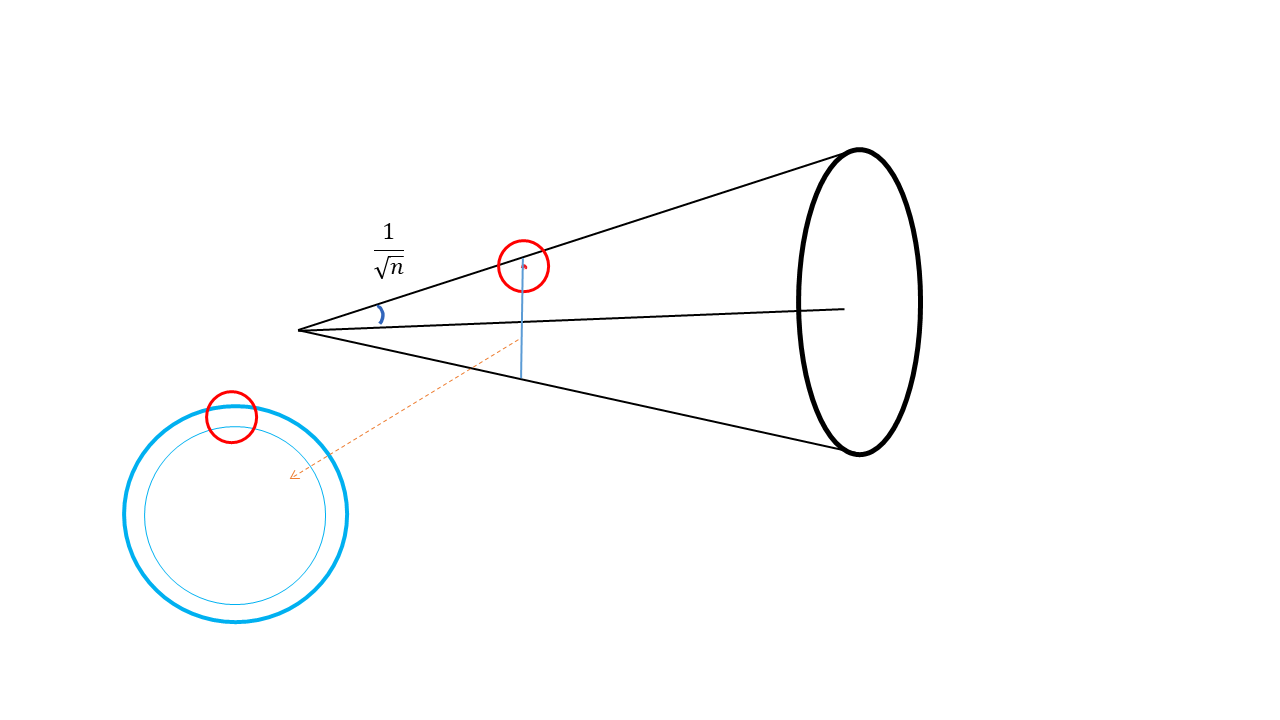}
\par\end{centering}
\centering{}\caption{The lower bound construction}
\end{figure}
\begin{proof}
W now prove the claim. We need to argue that starting at the distribution
at time $t$ in each slice is unimodal, i.e., radially monotonic nonincreasing
going out from the center. Consider two points $x,y$ on the same
slice with $x$ closer to the boundary. WLOG we can assume they are
on the same radial line from the center. Also,$\ell(x)\le\ell(y)$.
Let $B(x)=K_{D}\cap(x+\delta B)$. Every point $z\in B(y)\setminus B(x)$
has $\ell(z)\ge\ell(w)$ for any point in $w\in B(x)\setminus B(y)$.
Suppose the current distribution is $p_{t}$, which is a radially
monotonic nonincreasing function. Then, 
\begin{align*}
p_{t+1}(y)-p_{t+1}(x) & =\int_{z\in B(y)}\frac{p_{t}(z)}{\vol(\delta B)}dz+(1-\ell(y))p_{t}(y)-\int_{z\in B(x)}\frac{p_{t}(z)}{\vol(\delta B)}dz-(1-\ell(x))p_{t}(x)
\end{align*}

The above expression is minimized by choosing $p_{t}$ to be as uniform
as possible (under the constraint of monotonicity, i.e., $p_{t}(y)\ge p_{t}(x)$
and $\forall z\in B(y)\setminus B(x),w\in B(x)\setminus B(y),\,p_{t}(z)\ge p_{t}(w)$)
in particular we can set $p_{t}(z)=p_{t}$ to be constant in $B(x)\cup B(y)$.
Then,
\begin{align*}
p_{t+1}(y)-p_{t+1}(x) & \ge p_{t}\ell(y)+(1-\ell(y))p_{t}-p_{t}\ell(x)-(1-\ell(x))p_{t}\\
 & =0.
\end{align*}

Thus, the new density maintains monotonicity as claimed. 
\end{proof}

\begin{acknowledgement*}
This work was supported in part by NSF Awards CCF-1563838, CCF-1717349
and CCF-1740551. We thank Ronen Eldan, Ravi Kannan for inspiring discussions,
and Ronen for his wonderful invention of stochastic localization.
Finally, we thank Emanuel Milman for providing the proof for small
ball probability.
\end{acknowledgement*}
\bibliographystyle{plain}
\bibliography{../acg}

\end{document}